\documentclass[12pt]{article}

\usepackage{amssymb,a4}
\usepackage{amsmath,amsfonts,amssymb,amsthm}

\newcommand{\Z}{{\mathbb Z}}
\newcommand{\C}{{\mathbb C}}
\newcommand{\R}{{\mathbb R}}
\newcommand{\N}{{\mathbb N}}

\def\<{\langle}
\def\>{\rangle}

\newtheorem{thm}{Theorem}[section]
\newtheorem{prop}[thm]{Proposition}
\newtheorem{lem}[thm]{Lemma}

\newtheorem{rmk}[thm]{Remark}
\newtheorem{definition}[thm]{Definition}

\begin{document}

\begin{center}
{\Large \bf Twisted Heisenberg-Virasoro vertex operator algebra}
\end{center}

\begin{center}
{Hongyan Guo$^{a\; b}$ and
 Qing
Wang$^{a}$\footnote[1]{Partially supported by
 China NSF grants (Nos.11531004, 11622107), Natural Science Foundation of Fujian Province
(No. 2016J06002).}\\
$\mbox{}^{a}$School of Mathematical Sciences, Xiamen University, Xiamen 361005,
China\\
$\mbox{}^{b}$Department of Mathematical and Statistical Sciences, University of Alberta,

Edmonton T6G 2G1, Canada}
\end{center}

\begin{abstract}
In this paper, we study a new kind of vertex operator algebra related to the twisted Heisenberg-Virasoro algebra, which we call the twisted Heisenberg-Virasoro vertex operator algebra, and its modules.
 Specifically,
we present some results concerning the relationship between the restricted module categories of twisted Heisenberg-Virasoro algebras of rank one and rank two
and several different kinds of module categories of their corresponding vertex algebras. We also study fully the structures of the twisted Heisenberg-Virasoro vertex operator algebra,
give a characterization
of it as a tensor product of two well-known vertex operator algebras, and solve the commutant problem.
\end{abstract}

\section{Introduction}
\def\theequation{1.\arabic{equation}}
\setcounter{equation}{0}

This paper mainly consists of two parts.
One is the relationship between ($\phi$-coordinated) modules of vertex algebras and restricted modules of two Lie algebras,
the results are good in that we get equivalence of categories of modules,
which may provide new ways of looking at the representation theory of the two Lie algebras as well as the
representation theory of the obtained vertex algebras.
The other is a fulfilled study of vertex operator algebras we obtained.
This kind of vertex operator algebra
 looks like the form of combining Heisenberg and Virasoro vertex operator algebras.
The structure theory and representation theory of vertex operator algebras coming from Heisenberg and Virasoro algebras are well-known and beautiful,
so it is inevitable to consider the corresponding theory of our vertex operator algebras to
prove or disprove our ideas.

The rank one twisted Heisenberg-Virasoro algebra $\mathcal{L}$
was first studied in the paper \cite{ACKP},
it is spanned by the elements
$L_{n},I_{n}$, $c_{1},c_{2},c_{3}$, $n\in\Z $, and the Lie bracket is given by (cf. \cite{B1})
\begin{eqnarray*}
[L_{m}, L_{n}]=(m-n)L_{m+n}+\delta_{m+n,0}\frac{m^{3}-m}{12}c_{1},
\end{eqnarray*}
\begin{eqnarray*}
[L_{m}, I_{n}]=-n I_{m+n}-\delta_{m+n,0}(m^{2}+m)c_{2},
\end{eqnarray*}
\begin{eqnarray*}
[I_{m}, I_{n}]=m\delta_{m+n,0}c_{3},\;\;\; [\mathcal{L}, c_{i}]=0,\; i=1,2,3.
\end{eqnarray*}
With $L_{n}\mapsto -t^{n+1}\frac{d}{dt}$ and $I_{n}\mapsto t^{n}$,
it is the universal central extension of the Lie algebra
of differential operators on a circle of order at most one:
$$\{ f(t)\frac{d}{dt}+g(t)\ | \ f(t),g(t)\in\C[t,t^{-1}] \}.$$
The highest weight modules of $\mathcal{L}$ when $c_{3}$ acts in a nonzero way have been studied in the paper \cite{ACKP},
 when it acts as zero have been studied in the paper \cite{B1}, and $\mathcal{L}$ also has its role in the representation theory of toroidal Lie algebra (\cite{B2}).
Recently, the authors in the paper \cite{AR} give a free field realization of the twisted Heisenberg?Virasoro algebra
and study the representation theory of it when $c_{3}$ acts as zero using vertex-algebraic methods and screening operators.
In our paper, we study the restricted modules of $\mathcal{L}$ using vertex algebra methods and formal variables,
we give a characterization of this type of modules via vertex algebras and corresponding ($\phi$-coordinated)  modules,
where our $c_{1}, c_{2}, c_{3}$ can act as any complex numbers.
And the results are also used to study the irreducible modules of the obtained vertex operator algebras.

In \cite{XLT}, the authors generalized the rank one to rank two case, and call the Lie algebra arising from 2-dimensional torus,
which here we denote it by $\mathcal{L}^{*}$.
More precisely, let $A=\C[t_1^{\pm1},t_2^{\pm1}]$ be the ring of Laurent polynomials in two variables
and $B$ be the set of skew derivations of $A$ spanned by the elements of the form
$$E_{m,n}=t_1^{m}t_2^{n}(n d_1-m d_2),$$
where $(m,n)\in\Z^2$,
and $d_1, d_2$ are degree derivations of $A$.
Set $L=A\oplus B$. Then $L$ becomes a Lie algebra under the Lie bracket relations
\begin{eqnarray*}
[t_{1}^{m}t_{2}^{n}, t_{1}^{r}t_{2}^{s}]=0;
\end{eqnarray*}
\begin{eqnarray*}
[t_{1}^{m}t_{2}^{n}, E_{r,s}]=(nr-ms)t_{1}^{m+r}t_{2}^{n+s};
\end{eqnarray*}
\begin{eqnarray*}
[E_{m,n},E_{r,s}]=(nr-ms)E_{m+r,n+s}.
\end{eqnarray*}

Let $L'$ be the derived Lie subalgebra of $L$. Then
$L'$ is perfect and has a universal central extension $\mathcal{L}^{*}$ with the following Lie bracket relations \cite{XLT}:
\begin{eqnarray*}
[t_{1}^{m}t_{2}^{n}, t_{1}^{r}t_{2}^{s}]=0;\;\;\;[K_{i},\mathcal{L}^{*}]=0,\; i=1,2,3,4;
\end{eqnarray*}
\begin{eqnarray*}
[t_{1}^{m}t_{2}^{n}, E_{r,s}]=(nr-ms)t_{1}^{m+r}t_{2}^{n+s}+\delta_{m+r,0}\delta_{n+s,0}(mK_{1}+nK_{2});
\end{eqnarray*}
\begin{eqnarray*}
[E_{m,n},E_{r,s}]=(nr-ms)E_{m+r,n+s} +\delta_{m+r,0}\delta_{n+s,0}(mK_{3}+nK_{4}).
\end{eqnarray*}
where $(m,n),(r,s)\in \Z^{2}\backslash\{(0,0)\}$, $K_1,K_2,K_3,K_4$ are central elements.

In this paper, we give an association of the restricted modules of $\mathcal{L}^{*}$ with $\phi$-coordinated modules of
corresponding vertex algebra, where again our $K_{1},K_{2},K_{3},K_{4}$ can act
as arbitrary complex numbers.

In a series of papers, the authors use Lie algebras to construct and study vertex algebras, and they also
give the connections between the modules of the Lie algebras and the modules of the corresponding vertex algebras or their likes.
For example, the very beginning study of the association of affine and Virasoro algebras with vertex (operator) algebras (cf. \cite{FZ}),
and further studies like \cite{LL},\cite{Li2}, \cite{Li3}, \cite{DL}, etc.
Later on, many other Lie algebras, like toroidal Lie algebras, quantum torus Lie algebras,
deformed Heisenberg Lie algebras, Lie algebra $\mathfrak{gl}_{\infty}$, elliptic affine Lie algebra,
$q$-Virasoro algebra and unitary Lie algebra
have also been related to vertex algebras or their likes
(see \cite{FZ}, \cite{BBS}, \cite{LTW}, \cite{Li1}, \cite{Li3}, \cite{JL}, \cite{GLTW}, \cite{GW}).

As for the rank one twisted Heisenberg-Virasoro algebra $\mathcal{L}$, it contains both a
Heisenberg subalgebra and a Virasoro subalgebra. In the theory of vertex algebras, we usually write
the generating functions of Virasoro algebra and Heisenberg algebra as
$$L(x)=\sum\limits_{n\in\Z}L_{n}x^{-n-2},\;\; I(x)=\sum\limits_{n\in\Z}I_{n}x^{-n-1},$$
so we first consider these types of generating functions. After writing
the bracket relations in terms of generating functions, we see that the subset
which consists of $L(x)$ and $I(x)$, when acting on a restricted module $W$ of $\mathcal{L}$,
form a local subset (see \cite{Li3}, cf. \cite{LL}).
 So conceptually, it generates
a vertex algebra with $W$ a module under a certain vertex operator operation.
And in this case, the explicit vertex algebra we needed is actually an induced
module constructed from $\mathcal{L}$.
In the process, we observe that when writing the generating functions as the form
$$\widetilde{L}(x)=\sum\limits_{n\in\Z}L_{n}x^{-n},\;\; \widetilde{I}(x)=\sum\limits_{n\in\Z}I_{n}x^{-n},$$
the subset that consists of $\widetilde{L}(x)$ and $\widetilde{I}(x)$, when acting on a restricted module $W$ of $\mathcal{L}$, also forms a local subset,
but under the vertex operator operation which was introduced through the study of quantum
vertex algebras and their corresponding modules (\cite{Li4}, \cite{Li5}), it generates
conceptually a vertex algebra with $W$ being its $\phi$-coordinated module.
And in this case, we use another Lie algebra which is isomorphic to $\mathcal{L}$ to construct
the vertex algebra we needed.

For the rank two twisted Heisenberg-Virasoro algebra $\mathcal{L}^{*}$,
we write the generating functions as
\begin{eqnarray*}
T_{m}(x)=\sum\limits_{n\in\Z}t_{1}^{m}t_{2}^{n}x^{-n},\;\; E_{m}(x)=\sum\limits_{n\in\Z}E_{m,n}x^{-n}.
\end{eqnarray*}
Form the subset
$U_{W}=\{ {\bf 1}_{W}\}\cup\{T_{m}(x),E_{m}(x)\ | \ m\in\Z\}$ with $W$ being a restricted module of $\mathcal{L}^{*}$,
this subset is also a local subset.
And under the context of \cite{Li4} or \cite{Li5}, it generates
 a vertex algebra with $W$ being its $\phi$-coordinated modules.
To associate a vertex algebra to $\mathcal{L}^{*}$ explicitly, we construct
a new affine type Lie algebra $\widehat{\mathfrak{L}^{*}}$ and showed that its induced module $V_{\widehat{\mathfrak{L}^{*}}}(\ell_{1234},0)$ is a
vertex algebra, where $\ell_{1234}\in\C$ (see section 4).
Furthermore, we prove the correspondence between the restricted modules of $\mathcal{L}^{*}$
and $\phi$-coordinated modules of the vertex algebra $V_{\widehat{\mathfrak{L}^{*}}}(\ell_{1234},0)$.

Beside the equivalence between the categories of these two Lie algebra modules and corresponding vertex algebra modules, we study
in detail the new vertex operator algebra $V_{\mathcal{L}}(\ell_{123},0)$ obtained in section 2,
give the characterization of all the irreducible  vertex operator algebra $V_{\mathcal{L}}(\ell_{123},0)$-modules,
 and then
we consider the Zhu's algebra, $C_{2}$-cofiniteness, rationality, regularity, unitary property
of it and its simple descendant, we also consider the commutant of Heisenberg vertex operator algebra in it,
in the process, we give a characterization of it as a tensor product of two other
vertex operator algebras which are equipped with non standard conformal vectors (See section 3 for detail).

Y.Zhu in \cite{Z} constructed an associative algebra $A(V)$ (nowadays it is called Zhu's algebra of $V$) for a general vertex operator algebra $V$ and established a 1-1
correspondence between irreducible representations of $V$ and irreducible representations of $A(V)$.
$C_{2}$-cofiniteness, rationality and regularity are three different but closely related notions, and Zhu's algebra
plays an important role in studying them, since for $V$ being $C_{2}$-cofinite,
rational and regular, $A(V)$ must be a finite dimensional algebra (c.f.\cite{Z}, \cite{DLM2}).
The notion of regularity (which deals with the complete reducibility of weak modules)
was first introduced in the paper \cite{DLM1},
it is a generalization of rationality (which was first introduced in the paper \cite{Z} and deals with the complete reducibility of admissible modules).
Regularity implies rationality by definition, and it was showed in \cite{Li7}
that any regular vertex operator algebra is $C_{2}$-cofinite,
$C_{2}$-cofiniteness and rationality is equivalent to regularity for CFT type vertex operator algebras
has been proved in \cite{ABD}. For our vertex operator algebra $V_{\mathcal{L}}(\ell_{123},0)$,
it is closely related to two kinds of vertex operator algebras, as one may expected, Virasoro
and Heisenberg vertex operator algebras, their $C_{2}$-cofiniteness, rationality and regularity
are well-known (c.f. \cite{DLM1}, \cite{FZ}, \cite{W}).
In section 3, we show in two different ways that the Zhu's algebra
of $V_{\mathcal{L}}(\ell_{123},0)$ is infinite-dimensional, and its simple descendant also turns out to be infinite-dimensional,
 which immediately give that
our $V_{\mathcal{L}}(\ell_{123},0)$ and its simple descendant are not $C_{2}$-cofinite,
not rational and not regular.

 The unitary property of vertex operator algebras has been studied
 in \cite{DLin}, and most well-known vertex operator algebras turn out to be unitary,
 our $V_{\mathcal{L}}(\ell_{123},0)$ is also proved to be unitary under certain conditions.
 Unitarity of a vertex operator algebra is important in that it is the first step that one may want
 to construct conformal nets from vertex operator algebras, where the construction of conformal nets and
 the construction of vertex operator algebras are expected to be equivalent in the sense that you may
 get one from the other.
 The commutant of a vertex subalgebra in a vertex algebra was introduced by Frenkel and Zhu in the paper \cite{FZ},
 it is a generalization of the coset construction considered by Kac-Peterson in representation theory (\cite{KP})
 and Goddard-Kent-Oliver in conformal field theory (\cite{G-K-O}). Describing the generators (or even basis) of a commutant is generally
 a non-trivial problem (c.f. \cite{DW}, \cite{JLin}), the authors in the paper \cite{LAL}
 reducing the problem of describing commutant in an appropriate category of vertex algebras to
a question in commutative algebra, which is a new viewpoint,
here we solve our problem by giving a characterization of our vertex operator algebra as
a tensor product of a Heisenberg
 vertex operator algebra (with nonstandard conformal vector) and a Virasoro vertex operator algebra (constructed using a new conformal vectors).

 This paper is organized as follows: In section 2,
 we first review the definition of rank one twisted Heisenberg-Virasoro algebra $\mathcal{L}$
 and define its restricted modules,
 then we prove that the category of restricted $\mathcal{L}$-modules is equivalent to the category of modules for a specific vertex algebra
and we also present the equivalence between restricted
$\mathcal{L}$-module category and
$\phi$-coordinated module category for certain vertex algebra.
In section 3, we specifically study the structure theory of vertex operator algebra $V_{\mathcal{L}}(\ell_{123},0)$.
In section 4, we study
the relationship between the restricted module category of the rank two twisted Heisenberg-Virasoro Lie algebra $\mathcal{L}^{*}$ and
$\phi$-coordinated module category for a vertex algebra which is constructed based on a new Lie algebra.

\section{Modules and $\phi$-coordinated modules}
\def\theequation{2.\arabic{equation}}
\setcounter{equation}{0}

In this section, we associate the rank one twisted Heisenberg-Virasoro algebra $\mathcal{L}$ with $V_{\mathcal{L}}(\ell_{123},0)$
in terms of vertex algebra with its module and $\phi$-coordinated modules.
More specifically, we show that there is a one-to-one correspondence between the restricted
$\mathcal{L}$-modules of {\em level} $\ell_{123}$
and modules for the vertex algebra $ V_{\mathcal{L}}(\ell_{123},0)$.
And also the category of
restricted $\mathcal{L}$-modules of level $\ell_{123}$ is equivalent to that of
$\phi$-coordinated modules for the vertex algebra $V_{\mathfrak{L}}(\ell_{123},0)$,
where $\mathfrak{L}$ is a Lie algebra that is isomorphic to $\mathcal{L}$.

Throughout this paper, we denote by $\mathbb{N}$, $\Z$, $\C$, $\mathbb{C}^{\times}$ the set of nonnegative integers,
integers, complex numbers, nonzero complex numbers respectively, and
the symbols $x,x_{1},x_{2}\dots $ denote mutually commuting independent formal variables. All vector spaces in this paper are
considered to be over  $\mathbb{C}$. For a vector space $U$, $U((x))$ is the vector space of lower
truncated integral power series in $x$ with coefficients
 in $U$, $U[[x]]$ is the vector space of nonnegative integral
 power series in $x$ with coefficients in $U$, and
$U[[x,x^{-1}]]$ is the vector space of doubly infinite integral
 power series in $x$ with coefficients in $U$ .

\subsection{Basic notions}

For later use, we know from \cite{LL} that
 \begin{eqnarray}
 (x_{1}-x_{2})^{m}(\frac{\partial}{\partial x_{2}})^{n}x_{1}^{-1}
 \delta(\frac{x_{2}}{x_{1}})=0                     \label{eq:2.1}
 \end{eqnarray}
 for $m>n$, $m,n\in\N$, where $\delta(\frac{x_{1}}{x_{2}})=\sum\limits_{n\in\Z}x_{1}^{n}x_{2}^{-n}$.

 For the definition of vertex (operator) algebra and its modules, we follow \cite{LL}.
 Let $W$ be a general vector space. Set \begin{eqnarray}
\mathcal{E}(W)=\mbox{Hom}(W , W((x)))\subset(\mbox{EndW})[[x,x^{-1}]].     \label{eq:2.2}
\end{eqnarray}
The identity operator on $W$, denoted by $\textbf{1}_{W}$, is a special
element of $\mathcal{E}(W).$

The following notion of locality was introduced in \cite{Li3}:
\begin{definition} {\em Formal series $a(x),b(x)\in\mathcal{E}(W)$ are said to be
                 {\em mutually local} if there exists a nonzero polynomial $(x_{1}-x_{2})^{k}$ with $k\in\N$ such that
                 \begin{eqnarray}
                 (x_{1}-x_{2})^{k}a(x_{1})b(x_{2})=(x_{1}-x_{2})^{k}b(x_{2})a(x_{1}).  \label{eq:2.3}
                 \end{eqnarray}
                 A subset (subspace) $U$ of $\mathcal{E}(W)$ is said to be {\em local} if any
                 $a(x),b(x)\in U$ are mutually local.}
\end{definition}

Recall the
basic notions and results on
$\phi$-coordinated modules for vertex algebras (\cite{Li4}).
Set
$$\phi=\phi(x,z)=xe^{z}\in\mathbb{C}((x))[[z]],$$
 which is fixed throughout the paper.

  \begin{definition} {\em Let $V$ be a vertex algebra.
               A {\em $\phi$-coordinated $V$-module} is a vector space $W$ equipped with a
               linear map
               $$Y_{W}(\cdot,x): V\longrightarrow \mathrm{Hom}(W , W((x)))\subset(\mathrm{End W})[[x,x^{-1}]],$$
               satisfying the conditions that $Y_{W}(\textbf{1},x)=\textbf{1}_{W}$
               and that for $u,v\in V$, there exists a nonzero polynomial $(x_{1}-x_{2})^{k}$ with $k\in\N$
               such that
               $$ (x_{1}-x_{2})^{k}Y_{W}(u,x_{1})Y_{W}(v,x_{2})\in \mathrm{Hom}(W,W((x_{1},x_{2})))$$
               and
               $$(x_{2}e^{z}-x_{2})^{k}Y_{W}(Y(u,z)v,x_{2})=((x_{1}-x_{2})^{k}Y_{W}(u,x_{1})Y_{W}(v,x_{2}))|_{x_{1}=x_{2}e^{z}}.$$}
\end{definition}

Let $W$ be a general vector space, $a(x),b(x)\in\mathcal{E}(W)$.
Assume that there exists a nonzero polynomial $p(x)$ such that
\begin{eqnarray}
p(x,z)a(x)b(z)\in \mathrm{Hom}(W,W((x,z))).  \qquad\;\label{eq:3.1}
\end{eqnarray}
Define $a(x)_{n}^{e}b(x)\in \mathcal{E}(W)$ for $n\in\mathbb{Z}$ in terms of
generating function
$$Y_{\mathcal{E}}^{e}(a(x),z)b(x) = \sum_{n\in\mathbb{Z}}(a(x)_{n}^{e}b(x))z^{-n-1}$$
by
$$Y_{\mathcal{E}}^{e}(a(x),z)b(x) = p(xe^{z},x)^{-1}(p(x_{1},x)a(x_{1})b(x))|_{x_{1}=xe^{z}}, $$
where $p(x_{1},x)$ is any nonzero polynomial such that (\ref{eq:3.1}) holds and $p(xe^{z},x)^{-1}$
stands for the inverse of $p(xe^{z},x)$ in $\mathbb{C}((x))((z)).$ (Note that $p(xe^{z},x)$ is a nonzero element in $\C((x))((z))$.)
The definition of $\phi$-coordinated module requires that $p(x,z)$ is of the form $(x-z)^{k}$ with $k\in\N$.

A subspace $U$ of $\mathcal{E}(W)$ such that every ordered pair
satisfies (\ref{eq:3.1}) is said to be {\em $Y_{\mathcal{E}}^{e}$-closed}
if $a(x)_{n}^{e}b(x)\in U$ for $a(x),b(x)\in U$, $n\in\mathbb{Z}$.
We denote by $\langle U\rangle_{e}$ the smallest $Y_{\mathcal{E}}^{e}$-closed
subspace of $\mathcal{E}(W)$ that contains $U$ and $\textbf{1}_{W}.$

The following result was obtained in \cite{Li4} (Theorem 5.4 and Proposition 5.3):

\begin{thm}\label{coordinated}
Let $U$ be a local subset of $\mathcal{E}(W)$. Then
  $(\langle U\rangle_{e},Y_{\mathcal{E}}^{e},\textbf{1}_{W})$ carries the structure of a vertex algebra and  $W$ is a
  $\phi$-coordinated $\langle U\rangle_{e}$-module with
 $Y_{W}(a(x),z) = a(z)$
  for $a(x)\in\langle U\rangle_{e}.$
\end{thm}

Now we are in a position to study the twisted Heisenberg-Virasoro algebra in terms of vertex algebra
with its module and its $\phi$-coordinated modules.

\subsection{Modules}

Firstly, we give the definition of
the rank one twisted Heisenberg-Virasoro algebra $\mathcal{L}$ (see \cite{ACKP} or \cite{B1}).

\begin{definition}\label{Twisted Heisenberg-Virasoro algebra}
{\em The rank one twisted Heisenberg-Virasoro algebra $\mathcal{L}$ is a Lie algebra with the basis
$\{L_{n},I_{n},c_{1},c_{2},c_{3}|n\in\Z  \}$, and the following Lie brackets:
\begin{eqnarray}
[L_{m}, L_{n}]=(m-n)L_{m+n}+\delta_{m+n,0}\frac{m^{3}-m}{12}c_{1},  \label{eq:2.4}
\end{eqnarray}
\begin{eqnarray}
[L_{m}, I_{n}]=-n I_{m+n}-\delta_{m+n,0}(m^{2}+m)c_{2},   \label{eq:2.5}
\end{eqnarray}
\begin{eqnarray}
[I_{m}, I_{n}]=m\delta_{m+n,0}c_{3},\;\;\; [\mathcal{L}, c_{i}]=0,\; i=1,2,3.   \label{eq:2.6}
\end{eqnarray}
}
\end{definition}

Clearly, $\mbox{Span}\{L_{n},\;c_{1}\ | \ n\in\Z\}$ is a Virasoro algebra, $\mbox{Span}\{I_{n},\;c_{3}\ | \ n\in\Z\}$ is
an infinite-dimensional Heisenberg algebra, we denote them by $Vir$, $\mathcal{H}$ respectively.

Form the generating functions as
$$L(x)=\sum\limits_{n\in\Z}L_{n}x^{-n-2},\;\; I(x)=\sum\limits_{n\in\Z}I_{n}x^{-n-1},$$
then the defining relations of $\mathcal{L}$ become to be
\begin{eqnarray}
&&{}[L(x_{1}), L(x_{2})]  \nonumber \\
&&{}=\sum\limits_{m,n\in\Z}(m-n)L_{m+n}x_{1}^{-m-2}x_{2}^{-n-2}+\sum\limits_{m\in\Z}\frac{m^{3}-m}{12}c_{1}x_{1}^{-m-2}x_{2}^{m-2} \nonumber \\
 &&{}=L^{'}(x_{2})x_{1}^{-1}\delta\left(\frac{x_{2}}{x_{1}}\right)
          + 2L(x_{2})\frac{\partial}{\partial x_{2}}x_{1}^{-1}\delta\left(\frac{x_{2}}{x_{1}}\right)
          +\frac{c_{1}}{12}\left(\frac{\partial}{\partial x_{2}}\right)^{3}x_{1}^{-1}\delta\left(\frac{x_{2}}{x_{1}}\right),    \label{eq:2.7}
 \end{eqnarray}

\begin{eqnarray}
&&{}[L(x_{1}), I(x_{2})]  \nonumber \\
&&{}=-\sum\limits_{m,n\in\Z}n I_{m+n}x_{1}^{-m-2}x_{2}^{-n-1}-\sum\limits_{m\in\Z}(m^{2}+m)c_{2}x_{1}^{-m-2}x_{2}^{m-1}  \nonumber \\
 &&{}=I^{'}(x_{2})x_{1}^{-1}\delta\left(\frac{x_{2}}{x_{1}}\right)
          + I(x_{2})\frac{\partial}{\partial x_{2}}x_{1}^{-1}\delta\left(\frac{x_{2}}{x_{1}}\right)
          -\left(\frac{\partial}{\partial x_{2}}\right)^{2}x_{1}^{-1}\delta\left(\frac{x_{2}}{x_{1}}\right)c_{2},    \label{eq:2.8}
 \end{eqnarray}
 \begin{eqnarray}
[I(x_{1}), I(x_{2})]
=\sum\limits_{m\in\Z}m c_{3}x_{1}^{-m-1}x_{2}^{m-1}
=\frac{\partial}{\partial x_{2}}x_{1}^{-1}\delta\left(\frac{x_{2}}{x_{1}}\right)c_{3},    \label{eq:2.9}
 \end{eqnarray}
 where $L^{'}(x)=\frac{d}{dx}(L(x)), I^{'}(x)=\frac{d}{dx}(I(x))$.

 We give the following definition.
 \begin{definition}\label{restricted and level}
{\em An $\mathcal{L}$-module $W$ is said to be {\em restricted} if for any $w\in W, n\in\Z$,
$L_{n}w=0$ and $I_{n}w=0$ for $n$ sufficiently large,
or equivalently, if $L(x), I(x)\in {\mathcal{E}}(W)$.
We say an $\mathcal{L}$-module $W$ is of {\em level} $\ell_{123}$ if the central element $c_{i}$ acts as
scalar $\ell_{i}$ for $i=1,2,3.$}
\end{definition}

Recall (\cite{LL}) that a Lie algebra $\mathfrak{g}$
equipped with a $\Z$-grading $\mathfrak{g}=\coprod_{n\in\Z}\mathfrak{g}_{(n)}$ is called a {\em $\Z$-graded Lie algebra} if
\begin{eqnarray*}
[\mathfrak{g}_{(m)},\mathfrak{g}_{(n)}]\subset\mathfrak{g}_{(m+n)}\;\;\mbox{for}\;m,n\in\Z.
\end{eqnarray*}
A subalgebra $\mathfrak{h}$
of a $\Z$-graded Lie algebra $\mathfrak{g}$
is called a {\em graded subalgebra} if
\begin{eqnarray*}
\mathfrak{h}=\coprod_{n\in\Z}\mathfrak{h}_{(n)},\;\mbox{where}\;\mathfrak{h}_{(n)}=\mathfrak{h}\cap \mathfrak{g}_{(n)}\;\mbox{for}\;n\in\Z.
\end{eqnarray*}
In particular, $\mathfrak{g}_{(0)}$, $\mathfrak{g}_{(\pm)}=\cup_{n\geq 1}\mathfrak{g}_{(\pm n)}$ and
$\mathfrak{g}_{(0)}\oplus \mathfrak{g}_{(\pm)}$ are graded subalgebras.

For the twisted Heisenberg-Virasoro algebra $\mathcal{L}$,
consider the following $\Z$-grading on $\mathcal{L}$
\begin{eqnarray}
\mathcal{L}=\coprod_{n\in\Z}\mathcal{L}_{(n)},
\end{eqnarray}
where
$$\mathcal{L}_{(0)}=\C L_{0}\oplus \C I_{0}\oplus \sum_{i=1}^{3}\C c_{i},\;\;\mbox{and}\;\;\mathcal{L}_{(n)}=\C L_{-n}\oplus \C I_{-n}\;\mbox{for}\;n\neq 0,$$
it makes $\mathcal{L}$ a $\Z$-graded Lie algebra, and this grading is given by ad $L_{0}$-eigenvalues.
Then we have the graded subalgebras
$$\mathcal{L}_{(-)}=\coprod_{n\geq 1}\mathcal{L}_{(-n)}=\coprod_{n\geq 1}\C L_{n}\oplus\coprod_{n\geq 1}\C I_{n},$$
$$\mathcal{L}_{(+)}=\coprod_{n\geq 1}\mathcal{L}_{(n)}=\coprod_{n\geq 1}\C L_{-n}\oplus\coprod_{n\geq 1}\C I_{-n}$$
and also $\mathcal{L}_{(0)}\oplus \mathcal{L}_{(-)}$ and $\mathcal{L}_{(0)}\oplus \mathcal{L}_{(+)}$.
 We also have the graded subalgebras
 \begin{eqnarray}
 \mathcal{L}_{(\leq 1)}=\coprod_{n\leq 1}\C L_{-n}\oplus\coprod_{n\leq 0}\C I_{-n}\oplus\sum_{i=1}^{3}\C c_{i},  \label{eq:2.10}
 \end{eqnarray}
 \begin{eqnarray}
 \mathcal{L}_{(\geq 2)}=\coprod_{n\geq 2}\C L_{-n}\oplus\coprod_{n\geq 1}\C I_{-n},        \label{eq:2.11}
 \end{eqnarray}
 and the decomposition
 \begin{eqnarray}
 \mathcal{L}=\mathcal{L}_{(\leq 1)}\oplus \mathcal{L}_{(\geq 2)}.                  \label{eq:2.12}
 \end{eqnarray}

Let $\ell_{i},i=1,2,3,$ be any complex numbers.
Consider $\C$ as an $\mathcal{L}_{(\leq 1)}$-module with $c_{i}$ acting
as the scalar $\ell_{i},i=1,2,3,$ and with $\coprod_{n\leq 1}\C L_{-n}\oplus\coprod_{n\leq 0}\C I_{-n}$
acting trivially. Denote this $\mathcal{L}_{(\leq 1)}$-module by $\C_{\ell_{123}}$.
Form the induced module
\begin{eqnarray}
V_{\mathcal{L}}(\ell_{123},0)=U(\mathcal{L})\otimes_{U(\mathcal{L}_{(\leq 1)})}\C_{\ell_{123}},   \label{eq:2.13}
\end{eqnarray}
where $U(\cdot)$ denotes the universal enveloping algebra of a Lie algebra.

Set ${\bf 1} =1\otimes 1\in V_{\mathcal{L}}(\ell_{123},0)$.
Define a linear operator $\overline{d}$ on $\mathcal{L}$
by
$$\overline{d}(c_{i})=0, \;\;\mbox{for}\; i=1,2,3,$$
$$\overline{d}(L_{n})=-(n+1)L_{n-1},\;\;\mbox{and}\;\; \overline{d}(I_{n})=-nI_{n-1}.$$
It is easy to check that $\overline{d}$ is a derivation of the twisted Heisenberg-Virasoro algebra $\mathcal{L}$, so
that $\overline{d}$ naturally extends to a derivation of the associative algebra $U(\mathcal{L})$.
Clearly, $\overline{d}$ preserves the subspace
$\coprod_{n\leq 1}\C L_{-n}\oplus\coprod_{n\leq 0}\C I_{-n} \oplus\sum\limits_{i=1}^{3}\C (c_{i}-\ell_{i})$
of $U(\mathcal{L})$.
Since as a (left) $U(\mathcal{L})$-module,
$$V_{\mathcal{L}}(\ell_{123},0)
\cong U(\mathcal{L})/U(\mathcal{L})
(\coprod_{n\leq 1}\C L_{-n}\oplus\coprod_{n\leq 0}\C I_{-n} \oplus\sum\limits_{i=1}^{3}\C (c_{i}-\ell_{i})),$$
it follows that $\overline{d}$ induces a linear operator on $V_{\mathcal{L}}(\ell_{123},0)$, which we denote by $d$.
Then we have
$$d({\bf 1})=0,\;\;\mbox{and}\;\; [d,L(x)]=\frac{d}{d x}L(x),\;\;[d,I(x)]=\frac{d}{d x}I(x).$$

From (\ref{eq:2.7}) to (\ref{eq:2.9}), using (\ref{eq:2.1}), we see that
\begin{eqnarray*}
(x_{1}-x_{2})^{4}[L(x_{1}), L(x_{2})] =0,\;\;(x_{1}-x_{2})^{3}[L(x_{1}), I(x_{2})] =0,\;\;(x_{1}-x_{2})^{2}[I(x_{1}), I(x_{2})] =0.
\end{eqnarray*}

By the Poincare-Birkhoff-Witt theorem, as a vector space we have
$$V_{\mathcal{L}}(\ell_{123},0)=U( \mathcal{L}_{(\geq 2)})\simeq S( \mathcal{L}_{(\geq 2)}).$$
And
\begin{eqnarray*}
V_{\mathcal{L}}(\ell_{123},0)=\coprod_{n\geq 0}V_{\mathcal{L}}(\ell_{123},0)_{(n)},
\end{eqnarray*}
where $V_{\mathcal{L}}(\ell_{123},0)_{(0)}=\C_{\ell_{123}}$ and $V_{\mathcal{L}}(\ell_{123},0)_{(n)}$, $n\geq 1$, has a basis consisting of the vectors
$$I_{-k_{1}}\cdots I_{-k_{s}}L_{-m_{1}}\cdots L_{-m_{r}}{\bf{1}} $$
for
$r, s\geq 0$,
$m_{1}\geq\cdots\geq m_{r}\geq 2$, $k_{1}\geq\cdots\geq k_{s}\geq 1$ with $\sum\limits_{i=1}^{r}m_{i}+\sum\limits_{j=1}^{s}k_{j}=n.$

Then by Theorem 5.7.1 of \cite{LL} we get that
\begin{thm}\label{VA1}
$V_{\mathcal{L}}(\ell_{123},0)$ is a vertex algebra,
which is uniquely determined by the condition that ${\bf 1}$ is the vacuum vector and
\begin{eqnarray}
Y(L_{-2}{\bf 1},x)=L(x)\;\left(=\sum\limits_{n\in\Z}L_{n}x^{-n-2}\right),       \label{eq:2.14}
\end{eqnarray}
\begin{eqnarray}
Y(I_{-1}{\bf 1},x)=I(x)\;\left(=\sum\limits_{n\in\Z}I_{n}x^{-n-1}\right).    \label{eq:2.15}
\end{eqnarray}
The vertex operator map $Y$ for this vertex algebra structure is given by
$$Y(I_{m_{1}}\cdots I_{m_{s}}L_{n_{1}}\cdots L_{n_{r}}{\bf 1},x)=I(x)_{m_{1}}\cdots I(x)_{m_{s}}L(x)_{n_{1}+1}\cdots L(x)_{n_{r}+1}{\bf 1}$$
for $r, s\geq 0$
and $n_{1},\cdots,n_{r},m_{1},\cdots,m_{s}\in\Z.$ Furthermore, $T=\{L_{-2}{\bf 1},I_{-1}{\bf 1}\}$ is the generating subset of
$V_{\mathcal{L}}(\ell_{123},0)$.
\end{thm}

{\bf Convention:} In our paper, for numbers of the form $n_{1},\ldots,n_{r}$, we say $r\geq 0$,
where $r=0$ means the element with subscript $n_{i}$'s do not appear.

In the following, we denote by $\omega=L_{-2}{\bf 1}$, $I=I_{-1}{\bf 1}$, and $\omega^{'}=\frac{1}{2\ell_{3}}I_{-1}I_{-1}{\bf 1}$,
note we have $\omega_{n}=L_{n-1}$, $(I)_{n}=(I_{-1}{\bf 1})_{n}=I_{n}$ (this is why we denote by $I_{-1}{\bf 1}$ the symbol $I$), for all $n\in\Z.$

\begin{rmk}\label{universal1}
{\em
As a module for the twisted Heisenberg-Virasoro algebra, $V_{\mathcal{L}}(\ell_{123},0)$ is generated by ${\bf 1}$ with the
relations $c_{i}=\ell_{i}$ and $L_{n}{\bf 1}=I_{m}{\bf 1}=0$ for $n\geq -1, m\geq 0,$ $i=1,2,3,$
and in fact $V_{\mathcal{L}}(\ell_{123},0)$ is {\em universal} in the sense that for any module $W$ of the
twisted Heisenberg-Virasoro algebra $\mathcal{L}$
of level $\ell_{123}$ equipped with a vector $e\in W$ such that $L_{n}e=I_{m}e=0$ for $n\geq -1, m\geq 0,$
there exists a unique $\mathcal{L}$-module homomorphism from $V_{\mathcal{L}}(\ell_{123},0)$
to $W$ sending ${\bf 1}$ to $e$.
}
\end{rmk}

We now use Theorem 5.5.18 and Theorem 5.7.6 of \cite{LL} to prove that:
\begin{thm}\label{Mod1}
Let $W$ be any restricted module for the rank one twisted Heisenberg-Virasoro algebra $\mathcal{L}$
of level $\ell_{123}.$ Then there exists a unique module structure on $W$ for
$V_{\mathcal{L}}(\ell_{123},0)$ viewed as a vertex algebra such that
\begin{eqnarray}
Y_{W}(L_{-2}{\bf 1},x)=L(x)\;\left(=\sum\limits_{n\in\Z}L_{n}x^{-n-2}\right),   \label{eq:2.16}
\end{eqnarray}
\begin{eqnarray}
Y_{W}(I_{-1}{\bf 1},x)=I(x)\;\left(=\sum\limits_{n\in\Z}I_{n}x^{-n-1}\right).       \label{eq:2.17}
\end{eqnarray}
The vertex operator map $Y_{W}$ for this module structure is given by
\begin{eqnarray}
Y_{W}(I_{m_{1}}\cdots I_{m_{s}}L_{n_{1}}\cdots L_{n_{r}}{\bf 1},x)
=I(x)_{m_{1}}\cdots I(x)_{m_{s}}L(x)_{n_{1}+1}\cdots L(x)_{n_{r}+1}{\bf 1}_{W} \label{eq:2.18}
\end{eqnarray}
for $r\geq 0, s\geq 0$ and $n_{1},\cdots,n_{r},m_{1},\cdots,m_{s}\in\Z.$
\end{thm}

\begin{proof}
Set
\begin{eqnarray*}
U_{W}=\{L(x) ,I(x),{\bf 1}_{W} \},
\end{eqnarray*}
then $U_{W}$ is a local subset of $\mathcal{E}(W)$, by Theorem 5.5.18 of \cite{LL},
$U_{W}$ generates a vertex algebra $\langle U_{W}\rangle$ with $W$ a natural faithful module.
Furthermore, $\langle U_{W}\rangle$ is the linear span of the elements of the form
$$a^{(1)}(x)_{n_{1}}\cdots a^{(r)}(x)_{n_{r}}{\bf 1}_{W}$$ for $a^{(i)}(x)\in U_{W},$ $n_{1},\ldots,n_{r}\in\Z$ with $r\geq 0.$
$\langle U_{W}\rangle$ is a $\mathcal{L}$-module
with $ L_{n}, I_{n}$ acting as $L(x)_{n+1}, I(x)_{n}$ for $n\in\Z$, so that $L(x)_{n}{\bf 1}_{W}=0$, $I(x)_{n}{\bf 1}_{W}=0$
for $n\geq 0$.
In view of Remark \ref{universal1}, there exists a unique $\mathcal{L}$-module map $\psi$ from
$V_{\mathcal{L}}(\ell_{123},0)$ to $\langle U_{W}\rangle$ such that $\psi({\bf 1})={\bf 1}_{W}$.
Then
\begin{eqnarray*}
\psi(\omega_{n}v)= L(x)_{n}\psi(v) \;\;\mbox{and}\; \psi(I_{n}v)= I(x)_{n}\psi(v),\;\;\mbox{for}\; n\in\Z, v\in V_{\mathcal{L}}(\ell_{123},0).
\end{eqnarray*}
The existence and uniqueness of $V_{\mathcal{L}}(\ell_{123},0)$-module structure on $W$ now immediately
follows from Theorem 5.7.6 of \cite{LL} with $T=\{\omega, I\}$.
\end{proof}

On the other hand, we have:
\begin{thm}\label{Mod2}
Every module $W$ for $V_{\mathcal{L}}(\ell_{123},0)$ viewed as a vertex algebra is naturally a
restricted module for the rank one twisted Heisenberg-Virasoro algebra $\mathcal{L}$ of level $\ell_{123}$,
with $L(x)=Y_{W}(L_{-2}{\bf 1},x)$, $I(x)=Y_{W}(I_{-1}{\bf 1},x)$.
\end{thm}
\begin{proof}
We have
$(L_{-2}{\bf 1})_{i}=\omega_{i}=L_{i-1}$, $(I_{-1}{\bf 1})_{i}=I_{i}$ for $i\in\Z.$
Then for $i\geq 0$,
\begin{eqnarray}
&&{} (L_{-2}{\bf 1})_{i}L_{-2}{\bf 1}=L_{i-1}L_{-2}{\bf 1}=[L_{i-1},L_{-2}]{\bf 1}  \nonumber\\
&&{} = (i+1)L_{i-3}
{\bf 1}+\delta_{i-3,0}\frac{(i-1)^{3}-(i-1)}{12}c_{1}{\bf 1},        \label{eq:2.19}
\end{eqnarray}
\begin{eqnarray}
&&{} (L_{-2}{\bf 1})_{i}I_{-1}{\bf 1}=L_{i-1}I_{-1}{\bf 1}=[L_{i-1},I_{-1}]{\bf 1}\nonumber\\
&&{}= I_{i-2}
{\bf 1}-\delta_{i-2,0}((i-1)^{2}+(i-1))c_{2}{\bf 1},        \label{eq:2.20}
\end{eqnarray}
\begin{eqnarray}
(I_{-1}{\bf 1})_{i}I_{-1}{\bf 1}=I_{i}I_{-1}{\bf 1}=[I_{i},I_{-1}]{\bf 1}= i\delta_{i-1,0}c_{3}{\bf 1}.        \label{eq:2.21}
\end{eqnarray}
By Proposition 5.6.7 of \cite{LL}, we get
\begin{eqnarray}
&&{} [Y_{W}(L_{-2}{\bf 1}, x_{1}), Y_{W}(L_{-2}{\bf 1}, x_{2})]  \nonumber\\
&&{} =\sum\limits_{i\geq 0}\frac{(-1)^{i}}{i!}Y_{W}((L_{-2}{\bf 1})_{i}L_{-2}{\bf 1}, x_{2})
\left(\frac{\partial}{\partial x_{1}}\right)^{i}x_{2}^{-1}\delta\left(\frac{x_{1}}{x_{2}}\right)  \nonumber\\
&&{} = Y_{W}(L_{-3}{\bf 1}, x_{2})x_{2}^{-1}\delta\left(\frac{x_{1}}{x_{2}}\right)  \nonumber\\
     &&{}\;\quad -2Y_{W}(L_{-2}{\bf 1}, x_{2})\left(\frac{\partial}{\partial x_{1}}\right)x_{2}^{-1}\delta\left(\frac{x_{1}}{x_{2}}\right)  \nonumber\\
     &&{}\;\quad -\frac{1}{12}\left(\frac{\partial}{\partial x_{1}}\right)^{3}x_{2}^{-1}\delta\left(\frac{x_{1}}{x_{2}}\right)c_{1}{\bf 1},             \label{eq:2.22}
\end{eqnarray}
\begin{eqnarray}
&&{} [Y_{W}(L_{-2}{\bf 1}, x_{1}), Y_{W}(I_{-1}{\bf 1}, x_{2})]  \nonumber\\
&&{} =\sum\limits_{i\geq 0}\frac{(-1)^{i}}{i!}Y_{W}((L_{-2}{\bf 1})_{i}I_{-1}{\bf 1}, x_{2})
\left(\frac{\partial}{\partial x_{1}}\right)^{i}x_{2}^{-1}\delta\left(\frac{x_{1}}{x_{2}}\right)  \nonumber\\
&&{} = Y_{W}(I_{-2}{\bf 1}, x_{2})x_{2}^{-1}\delta\left(\frac{x_{1}}{x_{2}}\right)  \nonumber\\
     &&{}\;\quad -Y_{W}(I_{-1}{\bf 1}, x_{2})\left(\frac{\partial}{\partial x_{1}}\right)x_{2}^{-1}\delta\left(\frac{x_{1}}{x_{2}}\right)  \nonumber\\
     &&{}\;\quad -\left(\frac{\partial}{\partial x_{1}}\right)^{2}x_{2}^{-1}\delta\left(\frac{x_{1}}{x_{2}}\right)c_{2}{\bf 1},             \label{eq:2.23}
\end{eqnarray}
\begin{eqnarray}
&&{} [Y_{W}(I_{-1}{\bf 1}, x_{1}), Y_{W}(I_{-1}{\bf 1}, x_{2})]  \nonumber\\
&&{} =\sum\limits_{i\geq 0}\frac{(-1)^{i}}{i!}Y_{W}((I_{-1}{\bf 1})_{i}I_{-1}{\bf 1}, x_{2})
\left(\frac{\partial}{\partial x_{1}}\right)^{i}x_{2}^{-1}\delta\left(\frac{x_{1}}{x_{2}}\right)  \nonumber\\
&&{} = -\left(\frac{\partial}{\partial x_{1}}\right)x_{2}^{-1}\delta\left(\frac{x_{1}}{x_{2}}\right)c_{3}{\bf 1},             \label{eq:2.24}
\end{eqnarray}

Note that
$$Y_{W}(L_{-3}{\bf 1},x)=Y_{W}(d(L_{-2}{\bf 1}),x)=\frac{d }{d x}Y_{W}(L_{-2}{\bf 1},x)$$
and
$$Y_{W}(I_{-2}{\bf 1},x)=Y_{W}(d(I_{-1}{\bf 1}),x)=\frac{d }{d x}Y_{W}(I_{-1}{\bf 1},x).$$
With Proposition 2.3.6 of \cite{LL} and the fact that
\begin{eqnarray*}
x_{2}^{-1}\delta\left(\frac{x_{1}}{x_{2}}\right)=x_{1}^{-1}\delta\left(\frac{x_{1}}{x_{2}}\right)=x_{1}^{-1}\delta\left(\frac{x_{2}}{x_{1}}\right),
\end{eqnarray*}
we see that $W$ is a $\mathcal{L}$-module of {\em level} $\ell_{123}$ with
$L(x)=Y_{W}(L_{-2}{\bf 1},x)$,
 and $I(x)=Y_{W}(I_{-1}{\bf 1},x)$ for $L_{-2},I_{-1}\in\mathcal{L}$.
Since $W$ is a $V_{\mathcal{L}}(\ell_{123},0)$-module,
by definition $Y_{W}(L_{-2}{\bf 1},x)$, $Y_{W}(I_{-1}{\bf 1},x)\in\mathcal{E}(W)$.
Therefore, $W$ is a restricted $\mathcal{L}$-module of level $\ell_{123}$.
\end{proof}

\subsection{$\phi$-coordinated modules}

We now show the results with respect to $\phi$-coordinated modules.
Modifying the generating functions of $\mathcal{L}$ by a shift as follows
$$\widetilde{L}(x)=\sum\limits_{n\in\Z}L_{n}x^{-n},\;\; \widetilde{I}(x)=\sum\limits_{n\in\Z}I_{n}x^{-n},$$
then the defining relations of $\mathcal{L}$ become to be
\begin{eqnarray}
&&{}[\widetilde{L}(x_{1}), \widetilde{L}(x_{2})]  \nonumber \\
&&{}=\sum\limits_{m,n\in\Z}(m-n)L_{m+n}x_{1}^{-m}x_{2}^{-n}+\sum\limits_{m\in\Z}\frac{m^{3}-m}{12}c_{1}x_{1}^{-m}x_{2}^{m}  \nonumber \\
 &&{}=\left(x_{2}\frac{\partial}{\partial x_{2}}\widetilde{L}(x_{2})\right)\delta\left(\frac{x_{2}}{x_{1}}\right)
          + 2\widetilde{L}(x_{2})\left(x_{2}\frac{\partial}{\partial x_{2}}\right)\delta\left(\frac{x_{2}}{x_{1}}\right)
                                 \nonumber\\
 &&{} \;\quad +\frac{c_{1}}{12}\left(x_{2}\frac{\partial}{\partial x_{2}}\right)^{3}\delta\left(\frac{x_{2}}{x_{1}}\right)
          -\frac{c_{1}}{12}\left(x_{2}\frac{\partial}{\partial x_{2}}\right)\delta\left(\frac{x_{2}}{x_{1}}\right),    \label{eq:3.2}
 \end{eqnarray}

\begin{eqnarray}
&&{}[\widetilde{L}(x_{1}), \widetilde{I}(x_{2})] \nonumber \\
&&{}=-\sum\limits_{m,n\in\Z}n I_{m+n}x_{1}^{-m}x_{2}^{-n}-\sum\limits_{m\in\Z}(m^{2}+m)c_{2}x_{1}^{-m}x_{2}^{m}   \nonumber \\
 &&{}=\left(x_{2}\frac{\partial}{\partial x_{2}}\widetilde{I}(x_{2})\right)\delta\left(\frac{x_{2}}{x_{1}}\right)
          + \widetilde{I}(x_{2})\left(x_{2}\frac{\partial}{\partial x_{2}}\right)\delta\left(\frac{x_{2}}{x_{1}}\right)
                                 \nonumber\\
  &&{}\;\quad -\left(x_{2}\frac{\partial}{\partial x_{2}}\right)^{2}\delta\left(\frac{x_{2}}{x_{1}}\right)c_{2}
          -\left(x_{2}\frac{\partial}{\partial x_{2}}\right)\delta\left(\frac{x_{2}}{x_{1}}\right)c_{2},    \label{eq:3.3}
 \end{eqnarray}
 \begin{eqnarray}
[\widetilde{I}(x_{1}), \widetilde{I}(x_{2})]
=\sum\limits_{m\in\Z}m c_{3}x_{1}^{-m}x_{2}^{m}
=\left(x_{2}\frac{\partial}{\partial x_{2}}\right)\delta\left(\frac{x_{2}}{x_{1}}\right)c_{3}.     \label{eq:3.4}
 \end{eqnarray}
We further set
$$\widehat{L}(x)=\widetilde{L}(x)-\frac{1}{24}c_{1},\;\; \widehat{I}(x)=\widetilde{I}(x)-c_{2},$$
then we have
\begin{eqnarray}
&&{}[\widehat{L}(x_{1}), \widehat{L}(x_{2})]  \nonumber \\
 &&{}=\left(x_{2}\frac{\partial}{\partial x_{2}}\widehat{L}(x_{2})\right)\delta\left(\frac{x_{2}}{x_{1}}\right)
          + 2\widehat{L}(x_{2})\left(x_{2}\frac{\partial}{\partial x_{2}}\right)\delta\left(\frac{x_{2}}{x_{1}}\right)
                                 \nonumber\\
 &&{} \;\quad +\frac{c_{1}}{12}\left(x_{2}\frac{\partial}{\partial x_{2}}\right)^{3}\delta\left(\frac{x_{2}}{x_{1}}\right),    \label{eq:3.5}
 \end{eqnarray}

\begin{eqnarray}
&&{}[\widehat{L}(x_{1}), \widehat{I}(x_{2})]  \nonumber \\
 &&{}=\left(x_{2}\frac{\partial}{\partial x_{2}}\widehat{I}(x_{2})\right)\delta\left(\frac{x_{2}}{x_{1}}\right)
          + \widehat{I}(x_{2})\left(x_{2}\frac{\partial}{\partial x_{2}}\right)\delta\left(\frac{x_{2}}{x_{1}}\right)
                                 \nonumber\\
  &&{}\;\quad -\left(x_{2}\frac{\partial}{\partial x_{2}}\right)^{2}\delta\left(\frac{x_{2}}{x_{1}}\right)c_{2},    \label{eq:3.6}
 \end{eqnarray}
 \begin{eqnarray}
[\widehat{I}(x_{1}), \widehat{I}(x_{2})]
=\left(x_{2}\frac{\partial}{\partial x_{2}}\right)\delta\left(\frac{x_{2}}{x_{1}}\right)c_{3}.     \label{eq:3.7}
 \end{eqnarray}

Now we need a new Lie algebra to establish the connection of $\mathcal{L}$ and vertex algebra by $\phi$-coordinated modules.

\begin{definition}\label{New Lie algebra 1}
{\em Let $\mathfrak{L}$ be a vector space spanned by the elements
$\overline{L}_{n}, \overline{I}_{n}, c_{i}$, $n\in\Z,i=1,2,3,$ we define the brackets of $\mathfrak{L}$ as follows:
\begin{eqnarray}
[\overline{L}_{m}, \overline{L}_{n}]=(m-n)\overline{L}_{m+n-1}
                                       +\frac{m(m-1)(m-2)}{12}\delta_{m+n-2,0}c_{1},  \label{eq:3.8}
\end{eqnarray}
\begin{eqnarray}
[\overline{L}_{m}, \overline{I}_{n}]=-n \overline{I}_{m+n-1}-(m^{2}-m)\delta_{m+n-1,0}c_{2},   \label{eq:3.9}
\end{eqnarray}
\begin{eqnarray}
[\overline{I}_{m}, \overline{I}_{n}]=m\delta_{m+n,0}c_{3},\;\;\; [\mathfrak{L}, c_{i}]=0, \;\; \mbox{for} \;\; i=1,2,3.   \label{eq:3.10}
\end{eqnarray}
}
\end{definition}

It is straightforward to see that $\mathfrak{L}$ is a Lie algebra, and it is isomorphic
to the rank one twisted Heisenberg-Virasoro algebra $\mathcal{L}$ via
$$\overline{L}_{m}\mapsto L_{m-1},\;\; \overline{I}_{m}\mapsto I_{m},\;\; c_i\mapsto c_i \;\; \mbox{for} \;\;i=1,2,3.$$

We set
$$\overline{L}(x)=\sum\limits_{n\in\Z}\overline{L}_{n}x^{-n-1},\;\;\quad \overline{I}(x)=\sum\limits_{n\in\Z}\overline{I}_{n}x^{-n-1},$$
then the definition relations of $\mathfrak{L}$ amount to
\begin{eqnarray}
&&{}[\overline{L}(x_{1}), \overline{L}(x_{2})]  \nonumber \\
&&{}=\sum\limits_{m,n\in\Z}(m-n)\overline{L}_{m+n-1}x_{1}^{-m-1}x_{2}^{-n-1}
         +\sum\limits_{m\in\Z}\frac{c_{1}}{12}m(m-1)(m-2)x_{1}^{-m-1}x_{2}^{m-3}    \nonumber \\
&&{}=\left(\frac{\partial}{\partial x_{2}}\overline{L}(x_{2})\right)x_{1}^{-1}\delta\left(\frac{x_{2}}{x_{1}}\right)
      +2\overline{L}(x_{2})\frac{\partial}{\partial x_{2}}x_{1}^{-1}\delta\left(\frac{x_{2}}{x_{1}}\right)
                             \nonumber \\
&&{} \;\quad +\frac{c_{1}}{12}\left(\frac{\partial}{\partial x_{2}}\right)^{3}x_{1}^{-1}\delta\left(\frac{x_{2}}{x_{1}}\right),   \label{eq:3.11}
\end{eqnarray}
\begin{eqnarray}
&&{}[\overline{L}(x_{1}), \overline{I}(x_{2})]   \nonumber\\
&&{}=-\sum\limits_{m,n\in\Z}n \overline{I}_{m+n-1}x_{1}^{-m-1}x_{2}^{-n-1}-\sum\limits_{m\in\Z}(m^{2}-m)c_{2}x_{1}^{-m-1}x_{2}^{m-2} \nonumber\\
&&{}=\left(\frac{\partial}{\partial x_{2}}\overline{I}(x_{2})\right)x_{1}^{-1}\delta\left(\frac{x_{2}}{x_{1}}\right)
           +\overline{I}(x_{2})\frac{\partial}{\partial x_{2}}x_{1}^{-1}\delta\left(\frac{x_{2}}{x_{1}}\right)
                             \nonumber \\
&&{}\;\quad-\left(\frac{\partial}{\partial x_{2}}\right)^{2}x_{1}^{-1}\delta\left(\frac{x_{2}}{x_{1}}\right)c_{2},  \label{eq:3.12}
\end{eqnarray}
\begin{eqnarray}
[\overline{I}(x_{1}), \overline{I}(x_{2})]
=\sum\limits_{m\in\Z}m c_{3}x_{1}^{-m-1}x_{2}^{m-1}
=\frac{\partial}{\partial x_{2}}x_{1}^{-1}\delta\left(\frac{x_{2}}{x_{1}}\right)c_{3}.     \label{eq:3.13}
\end{eqnarray}

Set
$$ \mathfrak{L}_{\geq 0}= \coprod_{n\geq 0}\overline{L}_{n}\oplus\coprod_{n\geq 0}\overline{I}_{n}
                     \oplus\sum\limits_{i=1}^{3}\C c_{i},\;\;\quad
    \mathfrak{L}_{< 0}= \coprod_{n< 0}\overline{L}_{n}\oplus\coprod_{n< 0}\overline{I}_{n}.$$
We see that $\mathfrak{L}_{\geq 0}$ and $\mathfrak{L}_{< 0}$ are Lie subalgebras
and $\mathfrak{L}=\mathfrak{L}_{\geq 0}\oplus\mathfrak{L}_{< 0}$ as a
vector space.
Let $\ell_{i}\in\C, i=1,2,3,$ and we denote by $\C_{\ell_{123}}=\C$ the one-dimensional
$\mathfrak{L}_{\geq 0}$-module with
$\coprod_{n\geq 0}\overline{L}_{n}\oplus\coprod_{n\geq 0}\overline{I}_{n}$ acting trivially and
$c_{i}$ acting as $\ell_{i}$ for $i=1,2,3.$
Form the induced module
$$V_{\mathfrak{L}}(\ell_{123},0)
= U(\mathfrak{L})\otimes_{U(\mathfrak{L}_{\geq 0})}\C_{\ell_{123}}. $$
Set ${\bf 1} =1\otimes 1\in V_{\mathfrak{L}}(\ell_{123},0)$.
Similarly, one can show that there exists a natural vertex algebra structure on $V_{\mathfrak{L}}(\ell_{123},0)$
with a linear operator $\overline{d}$ on $\mathfrak{L}$ defined
by
$$\overline{d}(c_{i})=0, \;\;\mbox{for}\;\; i=1,2,3,$$
$$\overline{d}(\overline{L}_{n})=-n\overline{L}_{n-1},\;\;\mbox{and}\;\; \overline{d}(\overline{I}_{n})=-n\overline{I}_{n-1},$$
and it is uniquely determined by the condition that ${\bf 1}$ is the vacuum vector and
\begin{eqnarray}
Y(\overline{L}_{-1}{\bf 1},x)=\overline{L}(x)\left(=\sum\limits_{n\in\Z}\overline{L}_{n}x^{-n-1}\right) ,  \label{eq:3.14}
\end{eqnarray}
\begin{eqnarray}
Y(\overline{I}_{-1}{\bf 1},x)=\overline{I}(x)\left(=\sum\limits_{n\in\Z}\overline{I}_{n}x^{-n-1}\right).    \label{eq:3.15}
\end{eqnarray}
The vertex operator map $Y$ for this vertex algebra structure is given by
$$Y(\overline{I}_{m_{1}}\cdots \overline{I}_{m_{s}}\overline{L}_{n_{1}}\cdots \overline{L}_{n_{r}}{\bf 1},x)
=\overline{I}(x)_{m_{1}}\cdots \overline{I}(x)_{m_{s}}\overline{L}(x)_{n_{1}}\cdots \overline{L}(x)_{n_{r}}{\bf 1}$$
for $r\geq 0,s\geq 0$ and $n_{1},\cdots,n_{r},m_{1},\cdots,m_{s}\in\Z.$
Furthermore, $T=\{\overline{L}_{-1}{\bf 1},\overline{I}_{-1}{\bf 1}\}$ is a generating subset of $V_{\mathfrak{L}}(\ell_{123},0)$.

\begin{rmk}\label{universal2}
{\em
As a module for the new Lie algebra $\mathfrak{L}$, $V_{\mathfrak{L}}(\ell_{123},0)$ is generated by ${\bf 1}$ with the
relations $c_{i}=\ell_{i}$ and $\overline{L}_{n}{\bf 1}=\overline{I}_{n}{\bf 1}=0$ for $n\geq 0,$ $i=1,2,3,$
and in fact $V_{\mathfrak{L}}(\ell_{123},0)$ is {\em universal} in the sense that for any module $W$ of $\mathfrak{L}$
of level $\ell_{123}$ equipped with a vector $e\in W$ such that $\overline{L}_{n}e=\overline{I}_{n}e=0$ for $n\geq 0,$
there exists a unique $\mathfrak{L}$-module homomorphism from $V_{\mathfrak{L}}(\ell_{123},0)$
to $W$ sending ${\bf 1}$ to $e$.
}
\end{rmk}

And we have:

\begin{thm} \label{coordinated mod1}
Let $W$ be a restricted $\mathcal{L}$-module of level $\ell_{123}$.
                  Then there exists a
                  $\phi$-coordinated $V_{\mathfrak{L}}(\ell_{123},0)$-module
                 structure $Y_{W}(\cdot,x)$ on $W$, which is uniquely
                 determined by
                 $$Y_{W}(\overline{L}_{-1}{\bf 1},x) = \widehat{L}(x)\;\;\mbox{and}
                            \;\; Y_{W}(\overline{I}_{-1}{\bf 1},x) = \widehat{I}(x)
                   \  \  \  \mbox{ for }\overline{L}_{-1},\overline{I}_{-1}\in\mathfrak{L}.$$
   The vertex operator map $Y_{W}$ for this module structure is given by
                   \begin{eqnarray*}
Y_{W}(\overline{I}_{m_{1}}\cdots \overline{I}_{m_{s}}\overline{L}_{n_{1}}\cdots \overline{L}_{n_{r}}{\bf 1},x)
=\widehat{I}(x)_{m_{1}}^{e}\cdots \widehat{I}(x)_{m_{s}}^{e}\widehat{L}(x)_{n_{1}}^{e}\cdots \widehat{L}(x)_{n_{r}}^{e}{\bf 1}_{W}
\end{eqnarray*}
for $r\geq 0, s\geq 0$ and $n_{1},\cdots,n_{r},m_{1},\cdots,m_{s}\in\Z.$
\end{thm}
\begin{proof}Since $T=\{\overline{L}_{-1}{\bf 1},\overline{I}_{-1}{\bf 1}\}$ generates $V_{\mathfrak{L}}(\ell_{123},0)$
as a vertex algebra, the uniqueness is clear. We now prove the existence.
Set $U_{W}=\{ {\bf 1}_{W}\}\cup\{\widehat{L}(x),\widehat{I}(x)\}\subset\mathcal{E}(W)$.
From (\ref{eq:3.5}) to (\ref{eq:3.7}), by using Lemma 2.1 of \cite{Li5}, we see that
\begin{eqnarray*}
(x_{1}-x_{2})^{4} [\widehat{L}(x_{1}),\widehat{L}(x_{2})]=0,\;\;
(x_{1}-x_{2})^{3}[\widehat{L}(x_{1}), \widehat{I}(x_{2})]=0,\;\;
(x_{1}-x_{2})^{2}[\widehat{I}(x_{1}), \widehat{I}(x_{2})]=0.
\end{eqnarray*}
Then as $U_{W}$ is a local subset of $\mathcal{E}(W)$, Theorem \ref{coordinated} tells us
 $U_{W}$ generates a vertex algebra $\langle U_{W}\rangle_{e}$
 under the vertex operator operation $Y_{\mathcal{E}}^{e}$ with $W$ a $\phi$-coordinated module,
 and $Y_{W}(a(x),z) = a(z)$ for $a(x)\in \langle U_{W}\rangle_{e}.$
 Using Lemma 4.13 or Proposition 4.14 of \cite{Li5} and (\ref{eq:3.5}), (\ref{eq:3.6}) and (\ref{eq:3.7}), we have
 $$\widehat{L}(x)_{i}^{e}\widehat{L}(x) = 0   \;\;\mbox{for}\; i= 2\;\mbox{and}\; i\geq 4,$$
 $$\widehat{L}(x)_{3}^{e}\widehat{L}(x) =\frac{\ell_{1}}{2}{\bf 1}_{W},$$
  $$\widehat{L}(x)_{1}^{e}\widehat{L}(x) =2\widehat{L}(x),$$
 $$\widehat{L}(x)_{0}^{e}\widehat{L}(x) =x\frac{\partial}{\partial x}\widehat{L}(x),$$
and
 $$\widehat{L}(x)_{i}^{e}\widehat{I}(x) = 0   \;\;\mbox{for}\; i\geq 3,$$
 $$\widehat{L}(x)_{2}^{e}\widehat{I}(x) =-2\ell_{2}{\bf 1}_{W},$$
  $$\widehat{L}(x)_{1}^{e}\widehat{I}(x) =\widehat{I}(x),$$
 $$\widehat{L}(x)_{0}^{e}\widehat{I}(x) =x\frac{\partial}{\partial x}\widehat{I}(x),$$
and
$$\widehat{I}(x)_{i}^{e}\widehat{I}(x) = 0   \;\;\mbox{for}\; i= 0\;\mbox{and}\; i\geq 2,$$
 $$\widehat{I}(x)_{1}^{e}\widehat{I}(x) =\ell_{3}{\bf 1}_{W}.$$
Then by Borcherds' commutator formula we have
\begin{eqnarray*}
&&[Y_{\mathcal{E}}^{e}(\widehat{L}(x),x_{1}),Y_{\mathcal{E}}^{e}(\widehat{L}(x),x_{2})] \nonumber\\
&=&\sum_{i\geq 0}Y_{\mathcal{E}}^{e}(\widehat{L}(x)_{i}^{e}\widehat{L}(x),x_{2})
   \frac{1}{i!}\left(\frac{\partial}{\partial x_{2}}\right)^{i}x_{1}^{-1}\delta\left(\frac{x_{2}}{x_{1}}\right)  \nonumber\\
&=&Y_{\mathcal{E}}^{e}(x\frac{\partial}{\partial x}\widehat{L}(x),x_{2})x_{1}^{-1}\delta\left(\frac{x_{2}}{x_{1}}\right)
    +2Y_{\mathcal{E}}^{e}(\widehat{L}(x),x_{2}) \frac{\partial}{\partial x_{2}} x_{1}^{-1}\delta\left(\frac{x_{2}}{x_{1}}\right)
                                  \nonumber\\
 &&{} + \frac{\ell_{1}}{12}{\bf 1}_{W}\left(\frac{\partial}{\partial x_{2}}\right)^{3} x_{1}^{-1}\delta\left(\frac{x_{2}}{x_{1}}\right) ,
 \end{eqnarray*}
 and
 \begin{eqnarray*}
&&[Y_{\mathcal{E}}^{e}(\widehat{L}(x),x_{1}),Y_{\mathcal{E}}^{e}(\widehat{I}(x),x_{2})] \nonumber\\
&=&\sum_{i\geq 0}Y_{\mathcal{E}}^{e}(\widehat{L}(x)_{i}^{e}\widehat{I}(x),x_{2})
   \frac{1}{i!}\left(\frac{\partial}{\partial x_{2}}\right)^{i}x_{1}^{-1}\delta\left(\frac{x_{2}}{x_{1}}\right)  \nonumber\\
&=&Y_{\mathcal{E}}^{e}(x\frac{\partial}{\partial x}\widehat{I}(x),x_{2})x_{1}^{-1}\delta\left(\frac{x_{2}}{x_{1}}\right)
    +Y_{\mathcal{E}}^{e}(\widehat{I}(x),x_{2}) \frac{\partial}{\partial x_{2}} x_{1}^{-1}\delta\left(\frac{x_{2}}{x_{1}}\right)
                                  \nonumber\\
 &&{}- \ell_{2}{\bf 1}_{W}\left(\frac{\partial}{\partial x_{2}}\right)^{2} x_{1}^{-1}\delta\left(\frac{x_{2}}{x_{1}}\right) ,
 \end{eqnarray*}
 and
 \begin{eqnarray*}
&&[Y_{\mathcal{E}}^{e}(\widehat{I}(x),x_{1}),Y_{\mathcal{E}}^{e}(\widehat{I}(x),x_{2})] \nonumber\\
&=&\sum_{i\geq 0}Y_{\mathcal{E}}^{e}(\widehat{I}(x)_{i}^{e}\widehat{I}(x),x_{2})
   \frac{1}{i!}\left(\frac{\partial}{\partial x_{2}}\right)^{i}x_{1}^{-1}\delta\left(\frac{x_{2}}{x_{1}}\right)  \nonumber\\
&=&\ell_{3}{\bf 1}_{W} \frac{\partial}{\partial x_{2}} x_{1}^{-1}\delta\left(\frac{x_{2}}{x_{1}}\right)  .
 \end{eqnarray*}

Comparing these with (\ref{eq:3.11}) to (\ref{eq:3.13}), we see that $\langle U_{W}\rangle_{e}$ is a $\mathfrak{L}$-module
of {\em level} $\ell_{123}$ with $\overline{L}(z),\overline{I}(z)$ acting as
$Y_{\mathcal{E}}^{e}(\widehat{L}(x),z),Y_{\mathcal{E}}^{e}(\widehat{I}(x),z)$ respectively, and
$\frac{\partial}{\partial z}\overline{L}(z),\frac{\partial}{\partial z}\overline{I}(z)$ acting as
$Y_{\mathcal{E}}^{e}(x\frac{\partial}{\partial x}\widehat{L}(x),z),$
$Y_{\mathcal{E}}^{e}(x\frac{\partial}{\partial x}\widehat{I}(x),z)$ respectively.
Since $\widehat{L}(x)_{n}^{e}{\bf 1}_{W}=\widehat{I}(x)_{n}^{e}{\bf 1}_{W} = 0$ for $n\neq-1,$
and $\widehat{L}(x)_{-1}^{e}{\bf 1}_{W}=\widehat{L}(x)$, $\widehat{I}(x)_{-1}^{e}{\bf 1}_{W}=\widehat{I}(x)$,
we have $\overline{L}_{n}{\bf 1}_{W}=\overline{I}_{n}{\bf 1}_{W}=0$ for $n\geq 0$.
Then it follows from Remark \ref{universal2} that there exists a unique
$\mathfrak{L}$-module homomorphism $\psi$ from $V_{\mathfrak{L}}(\ell_{123},0)$
to $\langle U_{W}\rangle_{e}$ with $\psi({\bf 1})={\bf 1}_{W}.$
So
$$\psi(\overline{L}_{-1}{\bf 1})=\widehat{L}(x)_{-1}^{e}{\bf 1}_{W}=\widehat{L}(x)\in\langle U_{W}\rangle_{e}, $$
$$\psi(\overline{I}_{-1}{\bf 1})=\widehat{I}(x)_{-1}^{e}{\bf 1}_{W}=\widehat{I}(x)\in\langle U_{W}\rangle_{e}. $$
Now for $s\in\Z$, $v\in V_{\mathfrak{L}}(\ell_{123},0),$ with (\ref{eq:3.14}), (\ref{eq:3.15}) we have
$$\psi((\overline{L}_{-1}{\bf 1})_{s}v)=\psi(\overline{L}_{s}v)=\widehat{L}(x)_{s}^{e}\psi(v)=\psi(\overline{L}_{-1}{\bf 1})_{s}^{e}\psi(v),$$
$$\psi((\overline{I}_{-1}{\bf 1})_{s}v)=\psi(\overline{I}_{s}v)=\widehat{I}(x)_{s}^{e}\psi(v)=\psi(\overline{I}_{-1}{\bf 1})_{s}^{e}\psi(v),$$
Since $T=\{\overline{L}_{-1}{\bf 1},\overline{I}_{-1}{\bf 1}\}$ generates $V_{\mathfrak{L}}(\ell_{123},0)$ as a vertex algebra,
it follows from Proposition 5.7.9 of \cite{LL}
that $\psi$ is a homomorphism of vertex algebra, and then $W$ becomes a $\phi$-coordinated module
of $V_{\mathfrak{L}}(\ell_{123},0)$ with
  $$Y_{W}(\overline{L}_{-1}{\bf 1},x) = \widehat{L}(x),\;\;
                            \;\; Y_{W}(\overline{I}_{-1}{\bf 1},x) = \widehat{I}(x)
                   \  \  \  \mbox{ for }\overline{L}_{-1},\overline{I}_{-1}\in\mathfrak{L}$$
and
\begin{eqnarray*}
Y_{W}(\overline{I}_{m_{1}}\cdots \overline{I}_{m_{s}}\overline{L}_{n_{1}}\cdots \overline{L}_{n_{r}}{\bf 1},x)
=\widehat{I}(x)_{m_{1}}^{e}\cdots \widehat{I}(x)_{m_{s}}^{e}\widehat{L}(x)_{n_{1}}^{e}\cdots \widehat{L}(x)_{n_{r}}^{e}{\bf 1}_{W}
\end{eqnarray*}
for $r\geq 0, s\geq 0$ and $n_{1},\cdots,n_{r},m_{1},\cdots,m_{s}\in\Z.$
\end{proof}

On the other hand, we have:
\begin{thm}\label{coordinated mod2}
Let $W$ be a $\phi$-coordinated $V_{\mathfrak{L}}(\ell_{123},0)$-module.
Then $W$ is a restricted $\mathcal{L}$-module of level $\ell_{123}$ with
  $\widehat{L}(x)=Y_{W}(\overline{L}_{-1}{\bf 1},x)$,
 and $\widehat{I}(x)=Y_{W}(\overline{I}_{-1}{\bf 1},x)$ for $\overline{L}_{-1},\overline{I}_{-1}\in\mathfrak{L}.$
 \end{thm}
\begin{proof}
For $\overline{L}_{-1},\overline{I}_{-1}\in\mathfrak{L}$,
 since $Y(\overline{L}_{-1}{\bf 1},x)=\overline{L}(x)$, $Y(\overline{I}_{-1}{\bf 1},x)=\overline{I}(x)$,
from the identities (\ref{eq:3.11}) to (\ref{eq:3.13}), by using (\ref{eq:2.1}) we see that
$$(x_{1}-x_{2})^{4}[Y(\overline{L}_{-1}{\bf 1},x_{1}),Y(\overline{L}_{-1}{\bf 1},x_{2})]=0,$$
$$(x_{1}-x_{2})^{3}[Y(\overline{L}_{-1}{\bf 1},x_{1}),Y(\overline{I}_{-1}{\bf 1},x_{2})]=0,$$
$$(x_{1}-x_{2})^{2}[Y(\overline{I}_{-1}{\bf 1},x_{1}),Y(\overline{I}_{-1}{\bf 1},x_{2})]=0.$$
Note that for $i\geq 0$, we have
\begin{eqnarray*}
&&{}(\overline{L}_{-1}{\bf 1})_{i}\overline{L}_{-1}{\bf 1}=\overline{L}_{i}\overline{L}_{-1}{\bf 1}=[\overline{L}_{i}, \overline{L}_{-1}]{\bf 1}\nonumber\\
&&{}=(i+1)\overline{L}_{i-2}{\bf 1}+\frac{i(i-1)(i-2)}{12}\delta_{i-3,0}\ell_{1}{\bf 1}.
\end{eqnarray*}
and
\begin{eqnarray*}
&&{}(\overline{L}_{-1}{\bf 1})_{i}\overline{I}_{-1}{\bf 1}=\overline{L}_{i}\overline{I}_{-1}{\bf 1}=[\overline{L}_{i}, \overline{I}_{-1}]{\bf 1}\nonumber\\
&&{}=\overline{I}_{i-2}{\bf 1}-(i^{2}-i)\delta_{i-2,0}\ell_{2}{\bf 1}.
\end{eqnarray*}
and
\begin{eqnarray*}
&&{}(\overline{I}_{-1}{\bf 1})_{i}\overline{I}_{-1}{\bf 1}=\overline{I}_{i}\overline{I}_{-1}{\bf 1}=[\overline{I}_{i}, \overline{I}_{-1}]{\bf 1}\nonumber\\
&&{}=i\delta_{i-1,0}\ell_{3}{\bf 1}.
\end{eqnarray*}
By Proposition 5.9 of \cite{Li4}, we have
\begin{eqnarray*}
&&{}[Y_{W}(\overline{L}_{-1}{\bf 1},x_{1}),Y_{W}(\overline{L}_{-1}{\bf 1},x_{2})] \nonumber\\
&&{}=\mbox{Res}_{x_{0}}x_{1}^{-1}\delta\left(\frac{x_{2}e^{x_{0}}}{x_{1}}\right)
    x_{2}e^{x_{0}}Y_{W}(Y(\overline{L}_{-1}{\bf 1},x_{0})\overline{L}_{-1}{\bf 1},x_{2})\nonumber\\
 &&{}=Y_{W}(\overline{L}_{-2}{\bf 1}, x_{2})\delta\left(\frac{x_{2}}{x_{1}}\right)
      +2Y_{W}(\overline{L}_{-1}{\bf 1},x_{2})\left(x_{2}\frac{\partial}{\partial x_{2}}\right)\delta\left(\frac{x_{2}}{x_{1}}\right)
                             \nonumber \\
 &&{} \;\quad +\frac{\ell_{1}{\bf 1}_{W}}{12}\left(x_{2}\frac{\partial}{\partial x_{2}}\right)^{3}\delta\left(\frac{x_{2}}{x_{1}}\right),
\end{eqnarray*}
and
\begin{eqnarray*}
&&{}[Y_{W}(\overline{L}_{-1}{\bf 1},x_{1}),Y_{W}(\overline{I}_{-1}{\bf 1},x_{2})] \nonumber\\
&&{}=\mbox{Res}_{x_{0}}x_{1}^{-1}\delta\left(\frac{x_{2}e^{x_{0}}}{x_{1}}\right)
    x_{2}e^{x_{0}}Y_{W}(Y(\overline{L}_{-1}{\bf 1},x_{0})\overline{I}_{-1}{\bf 1},x_{2})\nonumber\\
 &&{}=Y_{W}(\overline{I}_{-2}{\bf 1},x_{2})\delta\left(\frac{x_{2}}{x_{1}}\right)
          + Y_{W}(\overline{I}_{-1}{\bf 1}, x_{2})\left(x_{2}\frac{\partial}{\partial x_{2}}\right)\delta\left(\frac{x_{2}}{x_{1}}\right)
                                 \nonumber\\
  &&{}   \quad\;  -\left(x_{2}\frac{\partial}{\partial x_{2}}\right)^{2}\delta\left(\frac{x_{2}}{x_{1}}\right)\ell_{2}{\bf 1}_{W},
\end{eqnarray*}
and
\begin{eqnarray*}
&&{}[Y_{W}(\overline{I}_{-1}{\bf 1},x_{1}),Y_{W}(\overline{I}_{-1}{\bf 1},x_{2})] \nonumber\\
&&{}=\mbox{Res}_{x_{0}}x_{1}^{-1}\delta\left(\frac{x_{2}e^{x_{0}}}{x_{1}}\right)
    x_{2}e^{x_{0}}Y_{W}(Y(\overline{I}_{-1}{\bf 1},x_{0})\overline{I}_{-1}{\bf 1},x_{2})\nonumber\\
&&{}=\left(x_{2}\frac{\partial}{\partial x_{2}}\right)\delta\left(\frac{x_{2}}{x_{1}}\right)\ell_{3}{\bf 1}_{W},
\end{eqnarray*}
Note that for a $\phi$-coordinated module, by Lemma 3.7 of \cite{Li4},
$$Y_{W}(\overline{L}_{-2}{\bf 1},x)=Y_{W}(d(\overline{L}_{-1}{\bf 1}),x)=x\frac{\partial }{\partial x}Y_{W}(\overline{L}_{-1}{\bf 1},x)$$
and
$$Y_{W}(\overline{I}_{-2}{\bf 1},x)=Y_{W}(d(\overline{I}_{-1}{\bf 1}),x)=x\frac{\partial }{\partial x}Y_{W}(\overline{I}_{-1}{\bf 1},x).$$
 Then $W$ is a $\mathcal{L}$-module of {\em level} $\ell_{123}$ with
$\widehat{L}(x)=Y_{W}(\overline{L}_{-1}{\bf 1},x)$,
 and $\widehat{I}(x)=Y_{W}(\overline{I}_{-1}{\bf 1},x)$ for $\overline{L}_{-1},\overline{I}_{-1}\in\mathfrak{L}$.
Since $W$ is a $\phi$-coordinated $V_{\mathfrak{L}}(\ell_{123},0)$-module,
by definition $Y_{W}(\overline{L}_{-1}{\bf 1},x)$, $Y_{W}(\overline{I}_{-1}{\bf 1},x)\in\mathcal{E}(W)$.
Therefore, $W$ is a restricted $\mathcal{L}$-module of level $\ell_{123}$.
\end{proof}

\section{Structures on twisted Heisenberg-Virasoro vertex operator algebra $V_{\mathcal{L}}(\ell_{123},0)$ and its irreducible modules}
\def\theequation{3.\arabic{equation}}
\setcounter{equation}{0}

In this section, we first show that $V_{\mathcal{L}}(\ell_{123},0)$ is actually a vertex operator algebra
 and characterize its irreducible modules.
Then we study the structure theory of $V_{\mathcal{L}}(\ell_{123},0)$
and get the corresponding results of the simple vertex operator algebra coming from it,
including Zhu's algebra, rationality,
$C_{2}$-cofiniteness, regularity, unitarity and the commutant of the Heisenberg vertex operator algebra in them,
in the process, we get a result that $V_{\mathcal{L}}(\ell_{123},0)$ can be characterized as a tensor product vertex operator algebra.

\subsection{Vertex operator algebra $V_{\mathcal{L}}(\ell_{123},0)$ and its irreducible modules}

For a $\Z$-graded Lie algebra $\mathfrak{g}=\coprod_{n\in\Z}\mathfrak{g}_{(n)}$,
a {\em $\C$-graded $\mathfrak{g}$-module} is a $\mathfrak{g}$-module $W$ equipped with
a $\C$-grading $W=\coprod_{r\in\C}W_{(r)}$ such that
\begin{eqnarray}
\mathfrak{g}_{(n)}W_{(r)}\subset W_{(n+r)}\;\;\mbox{for}\;n\in\Z,\;r\in\C.
\end{eqnarray}

From the above section, we know that $V_{\mathcal{L}}(\ell_{123},0)$ is $\Z$-graded by $L_{0}$-eigenvalues,
and clearly, the two grading restriction conditions in the definition of vertex
operator algebra hold for $V_{\mathcal{L}}(\ell_{123},0)$.
It is also clear that $V_{\mathcal{L}}(\ell_{123},0)$ is a restricted $\mathcal{L}$-module, a restricted $Vir$-module (and also a restricted $\mathcal{H}$-module).
By Theorem 5.7.4 of \cite{LL},
in order to say that it has a vertex operator algebra structure,
it remains to check (5.7.22)
or (5.7.23), this is straightforward, so we have
\begin{prop}
$(V_{\mathcal{L}}(\ell_{123},0), Y, {\bf{1}}, {L_{-2}}{\bf{1}})$ is a vertex operator algebra with conformal vector $\omega=L_{-2}{\bf 1}$
and of central charge $\ell_{1}.$
\end{prop}

\begin{rmk}
{\em
$V_{\mathcal{L}}(\ell_{123},0)$ is generated by the conformal vector $\omega=L_{-2}{\bf{1}}$ and $I=I_{-1}{\bf{1}}$, it is not a
minimal vertex operator algebra since it has a proper vertex operator subalgebra $V_{Vir}(\ell_{1},0)$ (with the same conformal vector $\omega$).
}
\end{rmk}

We now investigate the modules of $V_{\mathcal{L}}(\ell_{123},0)$ viewed as a vertex operator algebra.
By Theorem \ref{Mod1}, if further $W$ is $\C$-graded by $L_{0}$-eigenvalues, then $W$ is a module for $V_{\mathcal{L}}(\ell_{123},0)$ viewed as a vertex operator algebra,
possibly without the two grading restrictions.

By Theorem \ref{Mod1} and Theorem \ref{Mod2} we have
\begin{thm}\label{Mod3}
 The modules for $V_{\mathcal{L}}(\ell_{123},0)$ viewed as a vertex operator algebra
(i.e. $\C$-graded by $L_{0}$-eigenvalues and with the two grading restrictions) are exactly those restricted modules
for the Lie algebra $\mathcal{L}$ of level $\ell_{123}$ that are $\C$-graded by $L_{0}$-eigenvalues and with the two grading restrictions.
Furthermore, for any $V_{\mathcal{L}}(\ell_{123},0)$-module $W$, the $V_{\mathcal{L}}(\ell_{123},0)$-submodules of $W$ are exactly
the submodules of $W$ for $\mathcal{L}$, and these submodules are in particular graded.
\end{thm}

Next, we will modify the construction of the $\mathcal{L}$-module $V_{\mathcal{L}}(\ell_{123},0)$ to get a certain natural family of restricted
$\mathcal{L}$-modules of level $\ell_{123}$ that are $\C$-graded by $L_{0}$-eigenvalues and satisfy the two grading restrictions.
Note such $\mathcal{L}$-modules are naturally modules for the vertex operator algebra $V_{\mathcal{L}}(\ell_{123},0)$ by the above theorem.

Consider $\C$ as an $\mathcal{L}_{(0)}$-module
with $c_{i}$ acting as the scalar $\ell_{i}$, $i=1, 2, 3$,
$L_{0}$ acting as $h_{1}$ and $I_{0}$ acting as $h_{2}$,
where $\ell_{1}, \ell_{2}, \ell_{3}, h_{1}, h_{2}\in\C$.

Let $\mathcal{L}_{(-)}$ acting trivially on $\C$, making
$\C$ an $(\mathcal{L}_{(-)}\oplus \mathcal{L}_{(0)})$-module, which we denote by $\C_{\ell_{123},h_{12}}$.
Form the induced module
\begin{eqnarray}
M_{\mathcal{L}}(\ell_{123},h_{1},h_{2})=U(\mathcal{L})\otimes_{U(\mathcal{L}_{(-)}\oplus \mathcal{L}_{(0)})}\C_{\ell_{123},h_{12}}.
\end{eqnarray}

Again from the Poincare-Birkhoff-Witt theorem, as a vector space we have
\begin{eqnarray}
M_{\mathcal{L}}(\ell_{123},h_{1},h_{2})=U(\mathcal{L}_{(+)})\otimes\C_{\ell_{123},h_{12}}
=U(\mathcal{L}_{(+)})\simeq S(\mathcal{L}_{(+)}).
\end{eqnarray}
We naturally consider $\C_{\ell_{123},h_{12}}$ as a subspace of $M_{\mathcal{L}}(\ell_{123},h_{1},h_{2})$ and we set
$${\bf{1}}_{(\ell_{123},h_{12})}=1\in \C_{\ell_{123},h_{12}}\subset M_{\mathcal{L}}(\ell_{123},h_{1},h_{2}).$$

Then
\begin{eqnarray*}
M_{\mathcal{L}}(\ell_{123},h_{1},h_{2})=\coprod_{n\geq 0}M_{\mathcal{L}}(\ell_{123},h_{1},h_{2})_{(n+h_{1})},
\end{eqnarray*}
where $M_{\mathcal{L}}(\ell_{123},h_{1},h_{2})_{(h_{1})}=\C_{\ell_{123},h_{12}}$
and $M_{\mathcal{L}}(\ell_{123},h_{1},h_{2})_{(n+h_{1})}$ for $n\geq 1$ is the $L_{0}$-eigenspace of eigenvalue $n+h_{1}$,
has a basis consisting of the vectors
$$I_{-k_{1}}\cdots I_{-k_{s}}L_{-m_{1}}\cdots L_{-m_{r}}{\bf{1}}_{(\ell_{123},h_{12})} $$
for $r$, $s\geq 0$,
 $m_{1}\geq\cdots\geq m_{r}\geq 1$, $k_{1}\geq\cdots\geq k_{s}\geq 1$ with $\sum\limits_{i=1}^{r}m_{i}+\sum\limits_{j=1}^{s}k_{j}=n.$
Hence, as with $V_{\mathcal{L}}(\ell_{123},0)$, $M_{\mathcal{L}}(\ell_{123},h_{1},h_{2})$, as a module for $\mathcal{L}$ of level $\ell_{123}$,
is $\C$-graded by $L_{0}$-eigenvalues (Note $\Z$-graded is clearly $\C$-graded).

Consequently, $M_{\mathcal{L}}(\ell_{123},h_{1},h_{2})$ with the given $\C$-grading satisfies
the grading restriction conditions.
This in particular implies that $M_{\mathcal{L}}(\ell_{123},h_{1},h_{2})$ is a restricted $\mathcal{L}$-module.

Thus from Theorem \ref{Mod3} we immediately have:
\begin{thm}
For any complex numbers $\ell_{1}, \ell_{2}, \ell_{3}, h_{1}, h_{2}$,
$W=M_{\mathcal{L}}(\ell_{123},h_{1},h_{2})$ has a unique module structure for the
vertex operator algebra $V_{\mathcal{L}}(\ell_{123},0)$ such that $Y_{W}(\omega,x)=L(x)$ and $Y_{W}(I_{-1}{\bf{1}},x)=I(x).$
\end{thm}

\begin{rmk}{\em
The $\mathcal{L}$-module $M_{\mathcal{L}}(\ell_{123},h_{1},h_{2})$ is commonly referred
to in the literature as the {\em Verma module} in the papers \cite{ACKP}, \cite{B1}).
As a module for $\mathcal{L}$, $M_{\mathcal{L}}(\ell_{123},h_{1},h_{2})$ is generated by ${\bf{1}}_{(\ell_{123},h_{12})}$
with the relations
\begin{eqnarray*}
 L_{0}{\bf{1}}_{(\ell_{123},h_{12})}=h_{1}{\bf{1}}_{(\ell_{123},h_{12})},
 \;I_{0}{\bf{1}}_{(\ell_{123},h_{12})}= h_{2}{\bf{1}}_{(\ell_{123},h_{12})},
 \; c_{i}=\ell_{i}, i=1,2,3,
\end{eqnarray*}
\begin{eqnarray*}
\mbox{and}\;\;L_{n} {\bf{1}}_{(\ell_{123},h_{12})}=0,
\; I_{n} {\bf{1}}_{(\ell_{123},h_{12})}=0\;\mbox{for}\; n\geq 1.
\end{eqnarray*}

$M_{\mathcal{L}}(\ell_{123},h_{1},h_{2})$ is {\em universal} in the sense that for any module $W$ for $\mathcal{L}$ of level $\ell_{123}$ equipped with
a vector $v$ such that $L_{0}v=h_{1}v, I_{0}v=h_{2}v, L_{n}v=0, I_{n}v=0$ for all $n\geq 1$, there exists a unique module map
$M_{\mathcal{L}}(\ell_{123},h_{1},h_{2})\longrightarrow W$ sending ${\bf{1}}_{(\ell_{123},h_{12})}$ to $v$.}
\end{rmk}



In general, $M_{\mathcal{L}}(\ell_{123},h_{1},h_{2})$ as a module for $\mathcal{L}$ may be reducible,
in which case it is a reducible $V_{\mathcal{L}}(\ell_{123},0)$-module by Theorem \ref{Mod3} (or Proposition 4.5.17 in \cite{LL}).
Since $M_{\mathcal{L}}(\ell_{123},h_{1},h_{2})_{(h_{1})}(=\C_{\ell_{123},h_{12}})$ generates $M_{\mathcal{L}}(\ell_{123},h_{1},h_{2})$,
for any proper submodule $U$, $U$ is graded by Theorem \ref{Mod3}, and
$$U_{(h_{1})}=U\cap M_{\mathcal{L}}(\ell_{123},h_{1},h_{2})_{(h_{1})}=0.$$
Let $T_{\mathcal{L}}(\ell_{123},h_{1},h_{2})$ be the sum of all the proper $\mathcal{L}$-submodules of $M_{\mathcal{L}}(\ell_{123},h_{1},h_{2})$, it is also graded.
Then $T_{\mathcal{L}}(\ell_{123},h_{1},h_{2})_{(h_{1})}=0$, and so $T_{\mathcal{L}}(\ell_{123},h_{1},h_{2})$ is also proper and is the largest
proper submodule. Set
\begin{eqnarray}
L_{\mathcal{L}}(\ell_{123},h_{1},h_{2})=M_{\mathcal{L}}(\ell_{123},h_{1},h_{2})/T_{\mathcal{L}}(\ell_{123},h_{1},h_{2}),
\end{eqnarray}
then $L_{\mathcal{L}}(\ell_{123},h_{1},h_{2})$ is an irreducible $\mathcal{L}$-module.

By Theorem \ref{Mod3}, $T_{\mathcal{L}}(\ell_{123},h_{1},h_{2})$ is also the (unique) largest proper $V_{\mathcal{L}}(\ell_{123},0)$-submodule
of $M_{\mathcal{L}}(\ell_{123},h_{1},h_{2})$, so that $L_{\mathcal{L}}(\ell_{123},h_{1},h_{2})$ is an irreducible $V_{\mathcal{L}}(\ell_{123},0)$-module.

\begin{thm}\label{irred}
For any complex numbers $\ell_{1},\ell_{2},\ell_{3},h_{1},h_{2}$, $L_{\mathcal{L}}(\ell_{123},h_{1},h_{2})$ is an irreducible
module for the vertex operator algebra $V_{\mathcal{L}}(\ell_{123},0)$.
Furthermore, the modules $L_{\mathcal{L}}(\ell_{123},h_{1},h_{2})$ for $h_{1}, h_{2}\in\C$ exhaust
those irreducible (vertex operator algebra) $V_{\mathcal{L}}(\ell_{123},0)$-modules.
\end{thm}
\begin{proof}
The first assertion has been showed above.
For the second assertion, let $W=\coprod_{r\in\C}W_{(r)}$ be an irreducible $V_{\mathcal{L}}(\ell_{123},0)$-module,
since $I_{0}$ is in the center, so it must acts on the irreducible module as a scalar, say
 $I_{0}$ acts on $W$ as a scalar $h_{2}$.
And by Theorem \ref{Mod3}, $W$ must be of level $\ell_{123}$, i.e. $c_{i}$ acts on $W$
as scalar $\ell_{i}$ for $i=1,2,3.$
From (4.1.22) of \cite{LL}, there exists $h_{1}\in\C$ such that $W_{(h_{1})}\neq 0$ and $W_{(h_{1}-n)}=0$ for all positive integers $n$.
Let $0\neq v\in W_{(h_{1})}.$ Then $L_{0}v=h_{1}v$ and
$L_{n}v=0$, $I_{n}v=0$ for $n\geq 1$ since $L_{n}v,I_{n}v\in W_{(h_{1}-n)}$.
Hence by the universal property of $M_{\mathcal{L}}(\ell_{123},h_{1},h_{2})$, there is a unique $\mathcal{L}$-module homomorphism
$$\psi:M_{\mathcal{L}}(\ell_{123},h_{1},h_{2})\longrightarrow W;\;{\bf{1}}_{(\ell_{123},h_{12})}\mapsto v.$$
By Proposition 4.5.1 of \cite{LL}, $\psi$ is a $V_{\mathcal{L}}(\ell_{123},0)$-module
homomorphism (since $L_{-2}\bf{1}$ and $I_{-1}\bf{1}$ generates $V_{\mathcal{L}}(\ell_{123},0)$).
Since $W$ is irreducible and $T_{\mathcal{L}}(\ell_{123},h_{1},h_{2})$ is the (unique) largest proper submodule of $M_{\mathcal{L}}(\ell_{123},h_{1},h_{2})$,
it follows that $\psi(M_{\mathcal{L}}(\ell_{123},h_{1},h_{2}))=W$ and that $Ker\;\psi=T_{\mathcal{L}}(\ell_{123},h_{1},h_{2}).$
Thus $\psi$ reduces to a $V_{\mathcal{L}}(\ell_{123},0)$-module isomorphism from $L_{\mathcal{L}}(\ell_{123},h_{1},h_{2})$ onto $W$.
\end{proof}




\begin{rmk}
{\em Similarly as in the case of $V_{\mathcal{L}}(\ell_{123},0)$, one can show that
there is a vertex operator algebra structure on $V_{\mathfrak{L}}(\ell_{123},0)$, with $\overline{L}_{-1}{\bf{1}}$ a conformal vector,
\begin{eqnarray*}
V_{\mathfrak{L}}(\ell_{123},0)=\coprod_{n\geq 0}V_{\mathfrak{L}}(\ell_{123},0)_{(n)},
\end{eqnarray*}
where $V_{\mathfrak{L}}(\ell_{123},0)_{(0)}=\C_{\ell_{123}}$ and $V_{\mathfrak{L}}(\ell_{123},0)_{(n)}$, $n\geq 1$,
 has a basis consisting of the vectors
$$\overline{I}_{-k_{1}}\cdots \overline{I}_{-k_{s}}\overline{L}_{-m_{1}}\cdots \overline{L}_{-m_{r}}{\bf{1}} $$
for $r,s\geq 0$, $m_{1}\geq\cdots\geq m_{r}\geq 1$, $k_{1}\geq\cdots\geq k_{s}\geq 1$ with $\sum\limits_{i=1}^{r}(m_{i}+1)+\sum\limits_{j=1}^{s}k_{j}=n.$
And actually $V_{\mathfrak{L}}(\ell_{123},0)$ and $V_{\mathcal{L}}(\ell_{123},0)$ are isomorphic as vertex operator algebras.
}
\end{rmk}


In the following subsections,
we will fully study the structure theory of the vertex operator algebra $V_{\mathcal{L}}(\ell_{123},0)$ and its simple descendant.

\subsection{Zhu's algebra, $C_{2}$-cofiniteness, rationality and regularity}

Recall the following notions and see for example \cite{DLM1}, \cite{DLM2}, \cite{Z} for detail.
\begin{definition}\label{C2-cofin}
{\em A vertex operator algebra $V$ is called {\em $C_{2}$-cofinite} if $dim V/C_{2}(V)<\infty$, where $C_{2}(V)=\mbox{span}\{u_{-2}v\ | \ u,v\in V\}.$
}
\end{definition}

\begin{definition}\label{rational}
{\em A vertex operator algebra $V$ is called {\em rational} if every admissible module is a direct sum of simple admissible modules.
}
\end{definition}

\begin{definition}\label{regular}
{\em A vertex operator algebra $V$ is called {\em regular} if every weak module is a direct sum of simple ordinary modules.
}
\end{definition}

For any vertex operator algebra $V$, it's {\em Zhu's algebra} is defined to be $A(V)=V/O(V)$, where $O(V)$ is the subspace of $V$ spanned
by elements
$$\{Res_{z}(Y(a,z) \frac{(1+z)^{wt\; a}}{z^{2}}b) \ | \ a,b\in V, a\;\mbox{homogeneous} \},$$
note
$$Res_{z}(Y(a,z) \frac{(1+z)^{wt\; a}}{z^{2}}b)=\sum_{i\geq 0}\binom{wt\; a}{i}a_{i-2}b,$$
with the bilinear operation $*$ on $V$ defined by
$$a*b=Res_{z}(Y(a,z) \frac{(1+z)^{wt\; a}}{z}b)=\sum_{i\geq 0}\binom{wt\; a}{i}a_{i-1}b\;\;\mbox{for}\; a\; \mbox{homogeneous}.$$

For $v\in V$, we denote the image of $v$ in $A(V)$ by $[v]$, then
$$[a]*[b]=\sum_{i\geq 0}\binom{wt\; a}{i}[a_{i-1}b]\;\;\mbox{for}\; a\; \mbox{homogeneous},$$
it follows from \cite{Z} that $(A(V), *)$ is an associative algebra with identity $[{\bf 1}].$

For any $u\in V$, denote $o(u)=u_{wt\;u-1}$.
We do not recall the correspondence between $V$-module and $A(V)$-module here,
the readers can check \cite{Z} for detail. As we need, we write down the following results.
\begin{lem}\label{O}
For homogeneous elements $a,b\in V$, and $m\geq n\geq 0$,
\begin{eqnarray*}
Res_{z}(Y(a,z) \frac{(z+1)^{wt\; a+n}}{z^{2+m}}b)\in O(V).
\end{eqnarray*}
\end{lem}

\begin{lem}
\begin{enumerate}
\item $o(u)=0$ for any $u\in O(V)$;
\item $o(u*v)=o(u)o(v)$;
\item $[u]\in A(V)$ acts on $A(V)$-module corresponds to the $o(u)$-action on the corresponding $V$-module.
\end{enumerate}
\end{lem}

Since $I_{0}$ commutes with $L_{m}, I_{n}$ for any $m,n\in\Z$, so we have
$$I_{0}v=0,\;\;\mbox{hence}\; [I_{0}v]=[0],\;\;\mbox{for any}\;v\in V_{\mathcal{L}}(\ell_{123},0).$$

For any vertex operator algebra $V$ with conformal vector (denoted by) $\omega$,
$[\omega]$ is in the center of $A(V)$,
hence for our $V_{\mathcal{L}}(\ell_{123},0)$,
$[\omega]*[I]=[I]*[\omega]$, i.e. $[\omega]$ commutes with $[I]$.

\begin{thm}
There exists an isomorphism of associative algebras
\begin{eqnarray*}
\varphi: \C[x,y]\longrightarrow A(V_{\mathcal{L}}(\ell_{123},0));\;1\mapsto [{\bf 1}],\;x\mapsto [\omega],\;y\mapsto [I].
\end{eqnarray*}
\end{thm}
\begin{proof}
Let $A$ be the subalgebra of $A(V_{\mathcal{L}}(\ell_{123},0))$ generated by $[\omega]$ and $[I]$.
 For the existence of the above surjective algebra homomorphism, it suffices to show that for every homogeneous
$u\in V_{\mathcal{L}}(\ell_{123},0)$ with $wt\; u\geq 1$, $[u]\in A$. We show it by induction on $wt \;u$, note $wt\; u\geq 1$.
For $u\in V_{\mathcal{L}}(\ell_{123},0)$ with $wt\; u=1$, we only have one choice $u=I$, clearly $[u]=[I]\in A$.
Suppose for all homogeneous $u$ with $wt\; u\leq m-1$ we have $[u]\in A$, then for $u$
with $wt\; u=m$, we may assume that
$$u=I_{-k_{1}}\cdots I_{-k_{s}}L_{-m_{1}}\cdots L_{-m_{r}}{\bf 1} $$
for $r,s\geq 0$, $m_{1}\geq\cdots\geq m_{r}\geq 2$, $k_{1}\geq\cdots\geq k_{s}\geq 1$
with $\sum_{i=1}^{r}m_{i}+\sum_{j=1}^{s}k_{j}=m.$
If all $k_{j}$'s are zero, i.e.
$$u=L_{-m_{1}}\cdots L_{-m_{r}}{\bf 1}, $$
then one get from \cite{W} Lemma 4.1 that $[u]$ can be generated by $[\omega].$
Otherwise, denote
$$u^{'}=I_{-k_{2}}\cdots I_{-k_{s}}L_{-m_{1}}\cdots L_{-m_{r}}{\bf 1},$$
then $[u^{'}]$ is in $A(V_{\mathcal{L}}(\ell_{123},0))$ by induction, and $u=I_{-k_{1}} u^{'}$, $k_{1}\geq 1.$
By Lemma \ref{O},
take $a=I$, $n=0$, note $wt\; a=1$, then we have
\begin{eqnarray*}
[(I_{-n}+I_{-n-1})b]=[0]\;\;\mbox{for}\;b\;\mbox{homogeneous},\;n\geq 1.
\end{eqnarray*}
So without loss of generality,
we may assume $k_{1}=1$, i.e. $u=I_{-1}u^{'}$.
Then
$$[I]*[u^{'}]=[(I_{-1}+I_{0})u^{'}]=[I_{-1}u^{'}]+[I_{0}u^{'}]=[u]$$
with $[u^{'}], [I]\in A$ we get $[u]\in A$.

We now show that $\varphi$ is also injective.
Assume $0\neq f(x,y)\in Ker(\varphi)$, we can write it as $f(x,y)=\sum\limits_{f.s.}a_{mn}x^{m}y^{n}$ (f.s. means finite sum),
then $\sum\limits_{f.s.}a_{mn}[\omega]^{m}[I]^{n}=[0]$ in $A(V_{\mathcal{L}}(\ell_{123},0))$,
and so
$\sum\limits_{f.s.}a_{mn}[\omega]^{m}[I]^{n}$ acts trivially on any $A(V_{\mathcal{L}}(\ell_{123},0))$-modules
which corresponds to that $\sum_{f.s.}a_{mn}o(\omega)^{m}o(I)^{n}$ acts as zero on the bottom level of any $V_{\mathcal{L}}(\ell_{123},0)$-modules.
Note $o(\omega)=\omega_{1}=L_{0}$ and $o(I)=I_{0}$.
Recall for any $h_{1},h_{2}\in\C$,
 $M_{\mathcal{L}}(\ell_{123},h_{1},h_{2})$ is a module for $V_{\mathcal{L}}(\ell_{123},0)$,
with the bottom level $\C_{\ell_{123},h_{12}}$, and $L_{0}$ acts as $h_{1}$, $I_{0}$ acts as $h_{2}$,
 hence on $\C_{\ell_{123},h_{12}}$ we have
\begin{eqnarray*}
&&{} 0=\sum_{f.s.}a_{mn}o(\omega)^{m}o(I)^{n}  \\
&&{}=\sum_{f.s.}a_{mn}L_{0}^{m}I_{0}^{n} =\sum_{f.s.}a_{mn}h_{1}^{m}h_{2}^{n},
\end{eqnarray*}
but clearly, there exist elements $h_{1}, h_{2}$ such that $\sum\limits_{f.s.}a_{mn}h_{1}^{m}h_{2}^{n}\neq 0$,
contradiction, so $\varphi$ is injective.
\end{proof}

\begin{rmk}\label{assoc}
{\em Later on, we will characterize $V_{\mathcal{L}}(\ell_{123},0)$ as a tensor product of two vertex operator algebras,
which will immediately see that $A(V_{\mathcal{L}}(\ell_{123},0))$ is isomorphic to a polynomial algebra over $\C$
with two variables. But our proof above is more intrinsic and gives expectation (or some sense) that $V_{\mathcal{L}}(\ell_{123},0)$
maybe isomorphic to a tensor product of two vertex operator algebras.
}
\end{rmk}

Since for a vertex operator algebra to be regular, rational and $C_{2}$-cofinite,
its Zhu's algebra must be of finite dimensional (c.f. \cite{Z}, \cite{DLM2}),
now $A(V_{\mathcal{L}}(\ell_{123},0))\cong \C[x,y]$ is an infinite dimensional $\C$-algebra,
hence we have:
\begin{prop}
$V_{\mathcal{L}}(\ell_{123},0)$ is not regular, not rational and not $C_{2}$-cofinite.
\end{prop}

\subsection{Commutant}
We now look at the commutant of Heisenberg vertex operator algebra in $V_{\mathcal{L}}(\ell_{123},0)$.
First recall (\cite{FZ}) when $(V,Y,{\bf 1},\omega)$ is a vertex operator algebra and $(U,Y,{\bf 1},\omega^{'})$ is a vertex operator subalgebra of $V$,
 the {\em commutant} is defined to be
$$U^{c}=\{ v\in V \ | \ L^{'}(-1)v=0  \},$$
where $L^{'}(-1)$ is determined by $\omega^{'}.$

For our $V_{\mathcal{L}}(\ell_{123},0)$,
we know that when $\ell_{3}\neq 0$, $U=V_{\mathcal{H}}(\ell_{3},0)$ is a vertex operator algebra
with standard conformal vector $\omega^{'}=\frac{1}{2\ell_{3}}I_{-1}I_{-1}{\bf 1}$ (of central charge 1).
\begin{eqnarray*}
L_{1}\omega^{'}=\frac{1}{2\ell_{3}}L_{1}I_{-1}I_{-1}{\bf 1}  =\frac{-2\ell_{2}}{\ell_{3}}I_{-1}{\bf 1},
\end{eqnarray*}
so $L_{1}\omega^{'}=0$ if and only if $\ell_{2}=0$.
Hence when $\ell_{2}=0$, by Theorem 5.1 of \cite{FZ}, the commutant $U^{c}$ of the
Heisenberg vertex operator algebra $U=V_{\mathcal{H}}(\ell_{3},0)$ (equipped with the standard conformal vector $\omega^{'}$)
is a vertex operator algebra with conformal element $\omega^{''}=\omega-\omega^{'}.$

The next thing we want to do is to characterize the generators (or even basis)
of the commutant (actually, we get a more powerful result, see below).
We consider $\ell_{3}\neq 0$ in the following for necessity.

For the basic notions and results of tensor product vertex operator algebra, see for example \cite{LL}
The idea of the following theorem comes from the paper \cite{B2}.
\begin{thm}\label{tensor prod}
When $\ell_{3}\neq 0,$
$V_{\mathcal{L}}(\ell_{123},0)$ is isomorphic to the tensor product $V_{\mathcal{H}}(\ell_{3},0)\otimes V_{\tilde{Vir}}(c_{\tilde{Vir}},0)$
as vertex operator algebras,
where $V_{\mathcal{H}}(\ell_{3},0)$ is equipped with a nonstand conformal vector $\omega_{H}=\omega^{'}+\frac{\ell_{2}}{\ell_{3}}I_{-2}{\bf 1}$(of central charge $1-12\frac{\ell_{2}^{2}}{\ell_{3}}$),
and $\tilde{Vir}$ is a new Virasoro algebra constructed in the proof with central charge $c_{\tilde{Vir}}=\ell_{1}-1+12\frac{\ell_{2}^{2}}{\ell_{3}}.$
\end{thm}
\begin{proof}
Firstly, denote $\tilde{\omega}=\omega-\omega_{H}$,
it is straightforward to show that $\omega_{H}$ and $\tilde{\omega}$ satisfy Virasoro algebra relations.
So by the standard Virasoro algebra vertex operator algebra theory,
$\tilde{\omega}$ gives a Virasoro algebra which we denoted by $\tilde{Vir}$ and from $\tilde{Vir}$
we get a vertex operator algebra which we denoted by $V_{\tilde{Vir}}(c_{\tilde{Vir}},0)$
with conformal vector $\tilde{\omega}$ and is of central charge $c_{\tilde{Vir}}=\ell_{1}-1+12\frac{\ell_{2}^{2}}{\ell_{3}}.$
Since $(I_{-2}{\bf 1})_{0}=0$ and $(I_{-2}{\bf 1})_{1}=\omega^{'}_{1}-I_{0}$, $I_{0}$ acts on $V_{\mathcal{H}}(\ell_{3},0)$ as zero
 and $\omega^{'}$ is a conformal vector,
  we see that $\omega_{H}$ is also a conformal vector of $V_{\mathcal{H}}(\ell_{3},0)$ and is of central charge
$1-12\frac{\ell_{2}^{2}}{\ell_{3}}$.
Next, it is straightforward to show that
$$\tilde{\omega}_{n}I_{m}{\bf 1}=I_{m}\tilde{\omega}_{n}{\bf 1}$$
 for any $m,n\in\Z$.
At last, we define a map
$$\varphi:V_{\mathcal{H}}(\ell_{3},0)\otimes V_{\tilde{Vir}}(c_{\tilde{Vir}},0)\longrightarrow V_{\mathcal{L}}(\ell_{123},0)$$
on the basis by
$$I_{-k_{1}}\cdots I_{-k_{s}}{\bf 1}\otimes \tilde{\omega}_{-m_{1}}\cdots \tilde{\omega}_{-m_{r}}{\bf 1}
\mapsto I_{-k_{1}}\cdots I_{-k_{s}}\tilde{\omega}_{-m_{1}}\cdots \tilde{\omega}_{-m_{r}}{\bf 1}$$
and extend $\C$-linearly, it is easy to check that $\varphi$ is a linear isomorphism,
$\varphi({\bf 1}\otimes{\bf 1})={\bf 1}$ and $\varphi(\omega_{H}\otimes{\bf 1}+{\bf 1}\otimes\tilde{\omega})=\tilde{\omega}+\omega_{H}=\omega$,
it remains to show that it is a vertex algebra homomorphism, for this, note that
$$(I_{-k_{1}}\cdots I_{-k_{s}}\tilde{\omega}_{-m_{1}}\cdots \tilde{\omega}_{-m_{r}}{\bf 1})_{n}=\sum_{i\in\Z}(I_{-k_{1}}\cdots I_{-k_{s}}{\bf 1})_{i}(\tilde{\omega}_{-m_{1}}\cdots \tilde{\omega}_{-m_{r}}{\bf 1})_{n-i-1}$$
on $V_{\mathcal{L}}(\ell_{123},0)$ for all $n\in\Z$.
\end{proof}

\begin{rmk}
{\em
The equation $\tilde{\omega}_{n}I_{m}{\bf 1}=I_{m}\tilde{\omega}_{n}{\bf 1}$
is important in proving that $\varphi$ is a vertex algebra homomorphism.
Note when $\ell_{3}=0$, $V_{\mathcal{H}}(\ell_{3},0)$ still has a vertex algebra structure,
 take $\tilde{\omega}=\omega$,
so you may think just considering them as vertex algebras,
can we characterize $V_{\mathcal{L}}(\ell_{123},0)\cong V_{\mathcal{H}}(0,0)\otimes V_{Vir}(\ell_{1},0)$
as vertex algebras?
The answer is not positive since now we do not have $\omega_{n}I_{m}{\bf 1}=I_{m}\omega_{n}{\bf 1}$ in general
(even if you require $\ell_{2}=0$), the defined $\varphi$ is a linear isomorphism for sure, but it is not a vertex algebra homomorphism
anymore.
 }
\end{rmk}

By the above theorem, we get that
\begin{prop}
For any $\ell_{2}\in\C$, $\ell_{1}\in\C$, $0\neq\ell_{3}\in\C$, the commutant of $V_{\mathcal{H}}(\ell_{3},0)$ (with conformal vector $\omega_{H}$) in $V_{\mathcal{L}}(\ell_{123},0)$
 is a vertex operator algebra and
 is isomorphic to the Virasoro vertex operator algebra $V_{\tilde{Vir}}(c_{\tilde{Vir}},0)$
(whose structure is very clear).
\end{prop}

\begin{rmk}
{\em
Note in particular when $\ell_{2}=0$, $V_{\mathcal{H}}(\ell_{3},0)$ is with usual conformal vector $\omega_{H}=\omega^{'}$,
 so we've answered the commutant question that we originally considered.
}
\end{rmk}

\begin{rmk}
{\em As explained before in Remark \ref{assoc}, with Theorem \ref{tensor prod}, by the fact that Zhu's algebra $A(V\otimes W)=A(V)\otimes A(W),$
and the result of Heisenberg and Virasoro vertex operator algebras, we can immediately have that $A(V_{\mathcal{L}}(\ell_{123},0))$
 is isomorphic to a polynomial algebra in two variables.
 }
 \end{rmk}



\subsection{Simple vertex operator algebra and its structure}

We now look at the structure of the simple descendant of $V_{\mathcal{L}}(\ell_{123},0)$,
as usual case, the simple vertex operator algebra comes from $V_{\mathcal{L}}(\ell_{123},0)$
is of the form $L_{\mathcal{L}}(\ell_{123},0)=V_{\mathcal{L}}(\ell_{123},0)/T$,
where $T$ is the maximal ideal of $V_{\mathcal{L}}(\ell_{123},0)$
(which is equivalent to say that $T$ is the maximal $\mathcal{L}$-submodule of the $\mathcal{L}$-module $V_{\mathcal{L}}(\ell_{123},0)$),
we do not study the structure of the maximal submodule $T$ directly, instead we use Theorem \ref{tensor prod}
to get a characterization of it and hence of $V_{\mathcal{L}}(\ell_{123},0)/T$.



Since a vertex operator algebra is simple if it is simple as a vertex algebra,
simplicity of a vertex operator algebra does not depend on the conformal vector.
We know that $V_{\mathcal{H}}(\ell_{3},0)$ ($\ell_{3}\neq 0$) with the usual conformal vector $\omega^{'}$ is simple (of central charge 1),
so $V_{\mathcal{H}}(\ell_{3},0)$ with a nonstand conformal vector $\omega_{H}$
is also a simple vertex operator algebra (of central charge $1-12\frac{\ell_{2}^{2}}{\ell_{3}}$).

Denote by $c_{p,q}=1-6\frac{(p-q)^{2}}{pq}$, where $p,q\in\{2,3,4,\ldots\}$ and are relatively prime, recall (c.f. \cite{W}, \cite{LL})
\begin{prop}
\begin{enumerate}
\item If $c_{\tilde{Vir}}\neq c_{p,q}$, $V_{\tilde{Vir}}(c_{\tilde{Vir}},0)$ is a simple vertex operator algebra;
\item If $c_{\tilde{Vir}}= c_{p,q}$, it has the maximal ideal $\langle v_{p,q}\rangle$ generated by
  a singular vector $v_{p,q}$ of degree $(p-1)(q-1)$.
\end{enumerate}
\end{prop}

By Corollary 4.7.3  of \cite{FHL}, tensor product of vertex operator algebras is simple if and only if each term is simple.
So summary above we have
\begin{prop}
\begin{enumerate}
  \item When $\ell_{3}\neq 0$ and $c_{\tilde{Vir}}\neq c_{p,q}$, $V_{\mathcal{H}}(\ell_{3},0)\otimes V_{\tilde{Vir}}(c_{\tilde{Vir}},0)$
 is a simple vertex operator algebra, and hence $V_{\mathcal{L}}(\ell_{123},0)$ is also a simple vertex operator algebra.
  \item When $\ell_{3}\neq 0$ and $c_{\tilde{Vir}}= c_{p,q}$, $V_{\tilde{Vir}}(c_{\tilde{Vir}},0)$ has the maximal ideal generated by the singular vector $v_{p,q}$,
 and the quotient $V_{\tilde{Vir}}(c_{\tilde{Vir}},0)/\langle v_{p,q}\rangle$ is a simple vertex operator algebra with conformal vector
 $\overline{\tilde{\omega}}=\tilde{\omega}+\langle v_{p,q}\rangle$,
 hence then $V_{\mathcal{H}}(\ell_{3},0)\otimes (V_{\tilde{Vir}}(c_{\tilde{Vir}},0)/\langle v_{p,q}\rangle)$
 is a simple vertex operator algebra.
\end{enumerate}
\end{prop}

And we can prove
 \begin{prop}
 For $c_{\tilde{Vir}}= c_{p,q}$ and $\ell_{3}\neq 0$,
\begin{enumerate}
 \item $V_{\mathcal{H}}(\ell_{3},0)\otimes\langle v_{p,q}\rangle$ is a proper ideal of the vertex operator algebra
  $V_{\mathcal{H}}(\ell_{3},0)\otimes V_{\tilde{Vir}}(c_{\tilde{Vir}},0)$.
\item We have
$$(V_{\mathcal{H}}(\ell_{3},0)\otimes V_{\tilde{Vir}}(c_{\tilde{Vir}},0))/(V_{\mathcal{H}}(\ell_{3},0)\otimes\langle v_{p,q}\rangle)\cong
V_{\mathcal{H}}(\ell_{3},0)\otimes (V_{\tilde{Vir}}(c_{\tilde{Vir}},0)/\langle v_{p,q}\rangle)$$
as vertex operator algebras.
 Hence by the simplicity of latter we see that $V_{\mathcal{H}}(\ell_{3},0)\otimes\langle v_{p,q}\rangle$ is
 the maximal ideal. So the simple vertex operator algebra $L_{\mathcal{L}}(\ell_{123},0)$ is isomorphic to the tensor product $V_{\mathcal{H}}(\ell_{3},0)\otimes (V_{\tilde{Vir}}(c_{\tilde{Vir}},0)/\langle v_{p,q}\rangle)$.
\end{enumerate}
 \end{prop}
\begin{proof}
The first one is just using definition. For the second one, first construct a map
$$\pi:V_{\mathcal{H}}(\ell_{3},0)\otimes V_{\tilde{Vir}}(c_{\tilde{Vir}},0)\longrightarrow V_{\mathcal{H}}(\ell_{3},0)\otimes (V_{\tilde{Vir}}(c_{\tilde{Vir}},0)/\langle v_{p,q}\rangle)$$
by $v\otimes w\mapsto v\otimes\overline{w}$ and extend linearly, then clearly it is surjective with kernel
 $V_{\mathcal{H}}(\ell_{3},0)\otimes\langle v_{p,q}\rangle$,
and so induces a desired linear isomorphism, and also it is clearly that the induced map takes conformal vector to conformal vector and is
a vertex algebra homomorphism, hence is a vertex operator algebra isomorphism.
\end{proof}

Therefore the structure of the simple vertex operator algebra is clear:
\begin{thm}\label{simple}
\begin{enumerate}
  \item When $\ell_{3}\neq 0$ and $c_{\tilde{Vir}}\neq c_{p,q}$, $L_{\mathcal{L}}(\ell_{123},0)=V_{\mathcal{L}}(\ell_{123},0),$ so $T=0$.
  \item When $\ell_{3}\neq 0$ and $c_{\tilde{Vir}}= c_{p,q}$,
$$L_{\mathcal{L}}(\ell_{123},0)=V_{\mathcal{L}}(\ell_{123},0)/T\cong
  V_{\mathcal{H}}(\ell_{3},0)\otimes (V_{\tilde{Vir}}(c_{\tilde{Vir}},0)/\langle v_{p,q}\rangle),$$
  where $T=\varphi(V_{\mathcal{H}}(\ell_{3},0)\otimes\langle v_{p,q}\rangle)$ for $\varphi$ defined in Theorem \ref{tensor prod}.
\end{enumerate}
\end{thm}

Now for $\ell_{3}\neq 0$ and $c_{\tilde{Vir}}= c_{p,q}$,
 Zhu's algebra of the simple vertex operator algebra $L_{\mathcal{L}}(\ell_{123},0)$
is isomorphic to $A\left(V_{\mathcal{H}}(\ell_{3},0)\otimes (V_{\tilde{Vir}}(c_{\tilde{Vir}},0)/\langle v_{p,q}\rangle)\right)$
which is equal to $A(V_{\mathcal{H}}(\ell_{3},0))\otimes A(V_{\tilde{Vir}}(c_{\tilde{Vir}},0)/\langle v_{p,q}\rangle)$,
which then is isomorphic to $\C[x]\otimes(\C[y]/(G_{p,q}(y)))$ (where $(G_{p,q}(y))$ is a polynomial or degree $\frac{1}{2}(p-1)(q-1)$) (c.f. \cite{W}, \cite{FZ}),
 which is also of infinite dimensional,
so the simple vertex operator algebra $L_{\mathcal{L}}(\ell_{123},0)$ is also not regular, not $C_{2}$-cofinite and not rational.

\subsection{Unitary vertex operator algebra and its unitary module}
At the end of this section, we look at the unitary structure of the simple vertex operator algebra $L_{\mathcal{L}}(\ell_{123},0)$,
we follow some related notions by \cite{DLin}.
\begin{definition}\label{CFT}
{\em A vertex operator algebra $(V,Y,{\bf 1},\omega)$ is called {\em CFT-type}
if $V_{n}=0$, for $n<0$ and $V_{0}=\C {\bf 1}.$}
\end{definition}

Clearly, our $L_{\mathcal{L}}(\ell_{123},0)$ is of CFT-type.

\begin{definition}
{\em Let $(V,Y,{\bf 1},\omega)$ be a vertex operator algebra of CFT-type. An {\em anti-linear automorphism} $\phi$ of $V$ is an
anti-linear map $\phi:V\longrightarrow V$ such that $\phi({\bf 1})={\bf 1}$, $\phi(\omega)=\omega$ and
$\phi(u_{n}v)=\phi(u)_{n}\phi(v)$ for any $u,v\in V$ and $n\in\Z$.}
\end{definition}

\begin{definition}\label{unitary}
{\em Let $(V,Y,{\bf 1},\omega)$ be a vertex operator algebra of CFT-type and $\phi:V\longrightarrow V$ be an anti-linear involution,
i.e. an anti-linear automorphism of order 2. Then $(V,\phi)$ is called {\em unitary} if there exists a positive definite Hermitian
form $(,):V\times V\longrightarrow \C$ which is $\C$-linear on the first vector and anti-$\C$-linear on the second vector such that the following invariant
property holds: for any $a,u,v\in V$
\begin{eqnarray}
(Y(e^{zL(1)}(-z^{-2})^{L(0)}a,z^{-1})u,v)=(u,Y(\phi(a),z)v) \label{eq 3.5}
\end{eqnarray}
where $L(n)$ is defined by $Y(\omega,z)=\sum_{n\in\Z}L(n)z^{-n-2}.$}
\end{definition}

\begin{definition}\label{unitary voa mod}
{\em Let $(V,Y,{\bf 1},\omega)$ be a vertex operator algebra of CFT-type and $\phi:V\longrightarrow V$ be an anti-linear involution,
A $V$-module $(M,Y_{M})$ is called a {\em unitary $V$-module} if there exists a positive definite Hermitian
form $(,)_{M}:M\times M\longrightarrow \C$ which is $\C$-linear on the first vector and anti-$\C$-linear on the second vector such that the following invariant
property holds:
\begin{eqnarray}
(Y_{M}(e^{zL(1)}(-z^{-2})^{L(0)}a,z^{-1})w_{1},w_{2})_{M}=(w_{1},Y_{M}(\phi(a),z)w_{2})_{M} \label{eq 3.6}
\end{eqnarray}
for any $a\in V$, $w_{1},w_{2}\in M.$
}
\end{definition}

 \begin{rmk}\label{not isom inva}
{\em
Unitarity is not an isomorphic invariant property of vertex operator algebras, for example,
$V_{\mathcal{H}}(\ell_{3},0)$ is isomorphic to $V_{\mathcal{H}}(1,0)$ as vertex operator algebras
for any $0\neq\ell_{3}\in\C$, $V_{\mathcal{H}}(1,0)$ is a unitary vertex operator algebra (\cite{DLin}),
but of course it is not that for any $0\neq\ell_{3}\in\C$,
$V_{\mathcal{H}}(\ell_{3},0)$ is a unitary vertex operator algebra (for example, the positive definiteness of the Hermitian form requires $\ell_{3}\in\R_{> 0}$).
So even thought tensor product of unitary vertex operator algebras is unitary (Proposition 2.9 of \cite{DLin}),
and we know (Theorem \ref{tensor prod}, Theorem \ref{simple}) that our $L_{\mathcal{L}}(\ell_{123},0)$ is isomorphic to a tensor product of two
unitary vertex operator algebras under certain conditions (c.f. \cite{DLin})
we can't get that $V_{\mathcal{L}}(\ell_{123},0)$
is a unitary vertex operator algebra under those conditions.
}
\end{rmk}

Define an anti-linear map $\phi$ of $V_{\mathcal{L}}(\ell_{123},0)$ as follows:
$$\phi: V_{\mathcal{L}}(\ell_{123},0)\longrightarrow V_{\mathcal{L}}(\ell_{123},0),$$
$$I_{-k_{1}}\cdots I_{-k_{s}}L_{-m_{1}}\cdots L_{-m_{r}}{\bf 1}\mapsto (-1)^{s}I_{-k_{1}}\cdots I_{-k_{s}}L_{-m_{1}}\cdots L_{-m_{r}}{\bf 1},$$
where $r,s\geq 1$, $m_{1}\geq\cdots\geq m_{r}\geq 2$, $k_{1}\geq\cdots\geq k_{s}\geq 1$.

\begin{lem}
$\phi$ is an anti-linear involution of vertex operator algebra $V_{\mathcal{L}}(\ell_{123},0)$.
Furthermore, $\phi$ induces an anti-linear involution (also denoted by $\phi$) of $L_{\mathcal{L}}(\ell_{123},0)$.
\end{lem}

\begin{proof}
$\phi^{2}=id$, so it is enough to show that $\phi$ is an anti-linear automorphism.
By definition of our $\phi$ above, we have $\phi({\bf 1})={\bf 1}$, $\phi(\omega)=\omega$ and $\phi(I)=-I.$
Let
$$U=\{u\in V_{\mathcal{L}}(\ell_{123},0)\ | \ \phi(u_{n}v)=\phi(u)_{n}\phi(v), \forall\; v\in V_{\mathcal{L}}(\ell_{123},0),n\in\Z\},$$
then $U$ is a subspace of $V_{\mathcal{L}}(\ell_{123},0)$.
And it is straightforward to show that if $a,b\in U$, then $a_{m}b\in U$ for any $m\in\Z$.
Now we show that the generators $\omega$ and $I$ of $V_{\mathcal{L}}(\ell_{123},0)$ are in $U$,
$\forall\; v\in V_{\mathcal{L}}(\ell_{123},0),n\in\Z$,
$$\phi(\omega_{n}v)=\phi(L_{n-1}v)=L_{n-1}\phi(v)=\omega_{n}\phi(v)=(\phi(\omega))_{n}\phi(v),$$
$$\phi((I)_{n}v)=\phi(I_{n}v)=-I_{n}\phi(v)=\phi(I)_{n}\phi(v),$$
so $\omega,I\in U$, and then we have $V_{\mathcal{L}}(\ell_{123},0)=U$.
Thus $\phi$ is an anti-linear involution of $V_{\mathcal{L}}(\ell_{123},0)$.

Let $T$ be the maximal proper $\mathcal{L}$-submodule of $V_{\mathcal{L}}(\ell_{123},0)$, we have $\phi(T)$ is a proper $\mathcal{L}$-submodule
of $V_{\mathcal{L}}(\ell_{123},0)$, so $\phi(T)\subseteq T$, hence $\phi$ induces an anti-linear involution of $L_{\mathcal{L}}(\ell_{123},0)$.
\end{proof}

Note for the two Lie algebras $\widehat{\mathfrak{D}}$ and $\mathcal{L}$ defined in \cite{ACKP} and \cite{B1} respectively,
there exists a Lie algebra isomorphism between them (for notations related to $\widehat{\mathfrak{D}}$ we follow \cite{ACKP})
\begin{eqnarray*}
\rho:\mathcal{L}\longrightarrow \widehat{\mathfrak{D}};
\end{eqnarray*}
\begin{eqnarray*}
L_{m}\mapsto d_{m},\;c_{1}\mapsto c,\;c_{3}\mapsto c_{a},\;I_{n}\mapsto z^{n}-\sqrt{-1}\delta_{n,0}c_{3},\;c_{2}\mapsto-\sqrt{-1}c_{3},
\end{eqnarray*}
which then induces an isomorphism between the Verma modules (as $\mathcal{L}$-modules and as $\widehat{\mathfrak{D}}$-modules via $\rho$)
\begin{eqnarray*}
M_{\mathcal{L}}(\ell_{123},h_{1},h_{2})\cong\widetilde{R}(\ell_{1},h_{1},\ell_{3},h_{2}-\ell_{2},\sqrt{-1}\ell_{2}),
\end{eqnarray*}
and hence an isomorphism of irreducible $\mathcal{L}$-modules
\begin{eqnarray*}
L_{\mathcal{L}}(\ell_{123},h_{1},h_{2})\cong R(\ell_{1},h_{1},\ell_{3},h_{2}-\ell_{2},\sqrt{-1}\ell_{2}).
\end{eqnarray*}
In particular,
\begin{eqnarray*}
L_{\mathcal{L}}(\ell_{123},0)\cong R(\ell_{1},0,\ell_{3},-\ell_{2},\sqrt{-1}\ell_{2})
\end{eqnarray*}
as irreducible $\mathcal{L}$-modules,
where as irreducible $\mathcal{L}$-modules $L_{\mathcal{L}}(\ell_{123},0)=L_{\mathcal{L}}(\ell_{123},0,0).$

And via the isomorphism $\rho$, the anti-linear anti-involution $*$ defined in \cite{ACKP} induces
an anti-linear anti-involution $\sigma$ (see also \cite{B1}) on
 $\mathcal{L}$ defined as
\begin{eqnarray*}
&&{}\sigma(L_{n})=L_{-n},\;\;\sigma(I_{n})=I_{-n}-2\delta_{n,0}c_{2},\\
&&{}\sigma(c_{1})=c_{1},\;\;\sigma(c_{3})=c_{3},\;\;\sigma(c_{2})=-c_{2}.
\end{eqnarray*}

Then the unique contravariant (under $*$) Hermitian form in \cite{ACKP} induces a unique
 Hermitian form $(\;,\;)$ on
$L_{\mathcal{L}}(\ell_{123},0)$ such that $({\bf 1},{\bf 1})=1$ and is
contravariant with respect to the anti-involution $\sigma$, i.e.
\begin{eqnarray*}
(xu, v)=(u,\sigma(x)v),\;\;x\in\mathcal{L},\;\;u,v\in L_{\mathcal{L}}(\ell_{123},0),
\end{eqnarray*}
in particular
\begin{eqnarray*}
(I_{n}u, v)=(u,\sigma(I_{n})v),\;\;(L_{n}u,v)=(u,\sigma(L_{n})v),\;\;u,v\in L_{\mathcal{L}}(\ell_{123},0).
\end{eqnarray*}

Denote by $c_{m}=1-\frac{6}{m(m+1)},$, $h_{r,s}^{m}=\frac{(r(m+1)-sm)^{2}-1}{4m(m+1)}$, $1\leq s\leq r\leq m-1$, $ m\in\Z_{\geq 2}$.
Then for the positive definiteness of the Hermitian form,
by Theorem 6.6 of \cite{ACKP} we get that
\begin{prop}\label{pos-def}
The contravariant Hermitian form $(\;,\;)$ on
$L_{\mathcal{L}}(\ell_{123},h_{1},h_{2})$ is positive definite in precisely the following cases:
\begin{enumerate}
  \item $\ell_{3}=0$, then $\ell_{2}=0, h_{2}=0$ and $\ell_{1}\in\R_{\geq 1}, h_{1}\in\R_{\geq 0}$ or $\ell_{1}=c_{m}, h_{1}=h_{r,s}^{m}$, $1\leq s\leq r\leq m-1$, $ m\in\Z_{\geq 2}$;
  \item $\ell_{3}\in\R_{>0}$, then either
$$\ell_{1}-12\frac{\ell_{2}^{2}}{\ell_{3}}\in\R_{\geq 2},\;h_{1}-\frac{(h_{2}-\ell_{2})^{2}+(\sqrt{-1}\ell_{2})^{2}}{2\ell_{3}}\in\R_{\geq 0},$$
  or
$$\ell_{1}-12\frac{\ell_{2}^{2}}{\ell_{3}}=1+c_{m},\;
  h_{1}-\frac{(h_{2}-\ell_{2})^{2}+(\sqrt{-1}\ell_{2})^{2}}{2\ell_{3}}=h_{r,s}^{m},$$
where $1\leq s\leq r\leq m-1,$  $m\in\Z_{\geq 2}.$
\end{enumerate}
\end{prop}

For the unitary of $L_{\mathcal{L}}(\ell_{123},0)$ we need to show (\ref{eq 3.5}) holds and
we only need to show this for the generators $I,\omega$ of $L_{\mathcal{L}}(\ell_{123},0)$
by Proposition 2.11 of \cite{DLin}.
For $\omega$ we can show in the same way as Theorem 4.2 of \cite{DLin}, for $I$, we need do some more work since the anti-involution now is defined
differently.
For any $u,v\in L_{\mathcal{L}}(\ell_{123},0)$,
\begin{eqnarray*}
&&{}\left(Y(e^{zL_{1}}(-z^{-2})^{L_{0}}I,z^{-1})u,v  \right)  \\
&&{}= \left( Y(-z^{-2}I,z^{-1})u,v \right)
=\sum_{n\in\Z}-(I_{n}u,v)z^{n-1} \\
&&{}=\sum_{n\neq 0}-(I_{n}u,v)z^{n-1}+(-(I_{0}u,v)z^{-1}) \\
&&{}=\sum_{n\neq 0}-(u,I_{-n}v)z^{n-1}-(u,I_{0}v)z^{-1}-2\ell_{2}(u,v)z^{-1} \\
&&{}=-\sum_{n\in\Z}(u,I_{-n}v)z^{n-1}-2\ell_{2}(u,v)z^{-1} \\
&&{}=(u, Y(\phi(I),z)v)-2\ell_{2}(u,v)z^{-1},
\end{eqnarray*}
in order that (\ref{eq 3.5}) holds for $I$ we need to require that $\ell_{2}=0$.

The above shows that
\begin{thm}\label{unitary struc}
$(L_{\mathcal{L}}(\ell_{123},0),\phi)$ is a unitary vertex operator algebra if and only if one of the following holds:
\begin{enumerate}
  \item $\ell_{2}=0, \ell_{3}=0$, $\ell_{1}\in\R_{\geq 1}$ or $\ell_{1}=c_{m}, m\in\Z_{\geq 2}$;
  \item $\ell_{2}=0, \ell_{3}\in\R_{>0}$, $\ell_{1}\in\R_{\geq 2}$ or $\ell_{1}=1+c_{m}, m\in\Z_{\geq 2}.$
\end{enumerate}
\end{thm}

Similar as the prove of Theorem \ref{unitary struc}, we can get that
\begin{thm}\label{unitary mod}
Let $\phi$ be defined as above. Then $L_{\mathcal{L}}(\ell_{123},h_{1},h_{2})$ is a unitary module of
$(L_{\mathcal{L}}(\ell_{123},0),\phi)$ if and only if one of the following holds:
\begin{enumerate}
  \item $\ell_{2}=0,$ $\ell_{3}=0$, $h_{2}=0$, $\ell_{1}\in\R_{\geq 1}, h_{1}\in\R_{\geq 0}$ or $\ell_{1}=c_{m}, h_{1}=h_{r,s}^{m}$, $1\leq s\leq r\leq m-1$, $ m\in\Z_{\geq 2}$;
  \item $\ell_{2}=0,$ $\ell_{3}\in\R_{>0}$, $\ell_{1}\in\R_{\geq 2}$, $h_{1}-\frac{(h_{2})^{2}}{2\ell_{3}}\in\R_{\geq 0},$
  or $\ell_{1}=1+c_{m}$,
  $h_{1}-\frac{(h_{2})^{2}}{2\ell_{3}}=h_{r,s}^{m}$, $1\leq s\leq r\leq m-1,$  $m\in\Z_{\geq 2}.$
\end{enumerate}
\end{thm}

\begin{rmk}
{\em
\begin{enumerate}
  \item
Under the above conditions,
 $L_{\mathcal{L}}(\ell_{123},0)$ is another example of a unitary vertex operator algebra which
is neither rational nor $C_{2}$-cofinite.
\item By Theorem \ref{tensor prod} and Theorem \ref{simple}, we know that for any $\ell_{3}\neq 0$,
$V_{\mathcal{H}}(\ell_{3},0)$ (with $\omega_{H}$) is also a vertex operator subalgebra of
$L_{\mathcal{L}}(\ell_{123},0)$, and we can get the characterization of the commutant, its commutant
in the simple vertex operator algebra is a simple Virasoro vertex operator algebra, and
we know the condition for the unitary Virasoro vertex operator algebra (\cite{DLin}), but as mentioned in
the Remark \ref{not isom inva}, we can't say that the commutant of $V_{\mathcal{H}}(\ell_{3},0)$ (with $\omega_{H}$)
in $L_{\mathcal{L}}(\ell_{123},0)$ under those condition (which holds for the unitary Virasoro vertex operator algebra) is a unitary vertex operator algebra.
But we have the following:
\item When $\ell_{2}=0,$ $V_{\mathcal{H}}(\ell_{3},0)$ is equipped with the standard conformal vector $\omega^{'}$,
since our defined $\phi$ satisfies $\phi(\omega^{'})=\omega^{'}$, and $V_{\mathcal{H}}(\ell_{3},0)$ is a vertex operator subalgebra of
$L_{\mathcal{L}}(\ell_{123},0)$,
by Corollary 2.8 of \cite{DLin} we get that under the conditions in Theorem \ref{unitary struc},
the commutant $(U^{c},\phi|_{U^{c}})$ of $V_{\mathcal{H}}(\ell_{3},0)$ in $L_{\mathcal{L}}(\ell_{123},0)$ is also a unitary vertex operator algebra.
\end{enumerate}
}
\end{rmk}

\section{Rank two case with $\phi$-coordinated modules}
\def\theequation{4.\arabic{equation}}
\setcounter{equation}{0}

In this section, we associate the rank two twisted Heisenberg-Virasoro algebra $\mathcal{L}^{*}$ with the
vertex algebra $V_{\widehat{\mathfrak{L}^{*}}}(\ell_{1234},0)$
in terms of its $\phi$-coordinated module.
More specifically, we show that there is a one-to-one correspondence between the restricted
$\mathcal{L}^{*}$-module of {\em level} $\ell_{1234}$
and the $\phi$-coordinated module structure for the vertex algebra $V_{\widehat{\mathfrak{L}^{*}}}(\ell_{1234},0)$,
where $\mathfrak{L}^{*}$ is a new Lie algebra.

We first recall from \cite{XLT} the definition of
rank two twisted Heisenberg-Virasoro algebra.

\begin{definition}\label{Virasoro-like algebra}
{\em The rank two twisted Heisenberg-Virasoro algebra $\mathcal{L}^{*}$ is a vector space spanned by
$t_{1}^{m}t_{2}^{n}, E_{m,n},K_{i}, i=1,2,3,4$, for $m,n\in\Z^{2}\backslash\{(0,0)\}$,
where $K_{i}$ are central elements, with the following Lie brackets:
\begin{eqnarray*}
[t_{1}^{m}t_{2}^{n}, t_{1}^{r}t_{2}^{s}]=0;\;\;\;[K_{i}, \mathcal{L}^{*}]=0,\; i=1,2,3,4;
\end{eqnarray*}
\begin{eqnarray*}
[t_{1}^{m}t_{2}^{n}, E_{r,s}]=(nr-ms)t_{1}^{m+r}t_{2}^{n+s}+\delta_{m+r,0}\delta_{n+s,0}(mK_{1}+nK_{2});
\end{eqnarray*}
\begin{eqnarray*}
[E_{m,n},E_{r,s}]=(nr-ms)E_{m+r,n+s} +\delta_{m+r,0}\delta_{n+s,0}(mK_{3}+nK_{4}),
\end{eqnarray*}
for $(m,n) , (r,s)\in\Z^{2}\backslash\{(0,0)\}.$}
\end{definition}

We form the generating functions as
\begin{eqnarray}
T_{m}(x)=\sum\limits_{n\in\Z}t_{1}^{m}t_{2}^{n}x^{-n},\;\; E_{m}(x)=\sum\limits_{n\in\Z}E_{m,n}x^{-n}.    \label{eq:4.1}
\end{eqnarray}
Rewriting the Lie brackets as
\begin{eqnarray*}
[t_{1}^{m}t_{2}^{n}, E_{r,s}]=\left(n(m+r)-m(n+s)\right)t_{1}^{m+r}t_{2}^{n+s}+\delta_{m+r,0}\delta_{n+s,0}(mK_{1}+nK_{2}),
\end{eqnarray*}
\begin{eqnarray*}
[E_{m,n},E_{r,s}]=\left(n(m+r)-m(n+s)\right)E_{m+r,n+s} +\delta_{m+r,0}\delta_{n+s,0}(mK_{3}+nK_{4}),
\end{eqnarray*}
we can get the following generating function brackets:
\begin{eqnarray}
&&{} [T_{m}(x_{1}),E_{r}(x_{2})]\nonumber\\
&&{}=(m+r)T_{m+r}(x_{2})\left(x_{2}\frac{\partial}{\partial x_{2}}\right)\delta\left(\frac{x_{2}}{x_{1}}\right)
         + m\left(x_{2}\frac{\partial}{\partial x_{2}}T_{m+r}(x_{2})\right)\delta\left(\frac{x_{2}}{x_{1}}\right) \nonumber\\
&&{}\quad+m\delta_{m+r,0}\delta\left(\frac{x_{2}}{x_{1}}\right)K_{1}
           +\delta_{m+r,0}\left(x_{2}\frac{\partial}{\partial x_{2}}\right)\delta\left(\frac{x_{2}}{x_{1}}\right)K_{2},  \label{eq:4.2}
\end{eqnarray}
and
\begin{eqnarray}
&&{} [E_{m}(x_{1}), E_{r}(x_{2})]  \nonumber\\
&&{}=(m+r)E_{m+r}(x_{2})\left(x_{2}\frac{\partial}{\partial x_{2}}\right)\delta\left(\frac{x_{2}}{x_{1}}\right)
         + m\left(x_{2}\frac{\partial}{\partial x_{2}}E_{m+r}(x_{2})\right)\delta\left(\frac{x_{2}}{x_{1}}\right) \nonumber\\
&&{}\quad+m\delta_{m+r,0}\delta\left(\frac{x_{2}}{x_{1}}\right)K_{3}
           +\delta_{m+r,0}\left(x_{2}\frac{\partial}{\partial x_{2}}\right)\delta\left(\frac{x_{2}}{x_{1}}\right)K_{4}.  \label{eq:4.3}
\end{eqnarray}

\begin{definition}\label{restricted}
{\em An $\mathcal{L}^{*}$-module $W$ is said to be {\em restricted} if for any $w\in W, (m,n) , (r,s)\in\Z^{2}\backslash\{(0,0)\}$,
$t_{1}^{m}t_{2}^{n}w=0$ for $n$ sufficiently large and $E_{r,s}w=0$ for $s$ sufficiently large,
or equivalently, if $T_{m}(x)\in {\mathcal{E}}(W)$ and $E_{r}(x)\in {\mathcal{E}}(W)$ for $m,r\in\Z$.
We say an $\mathcal{L}^{*}$-module $W$ is of {\em level} $\ell_{1234}$ if the central element $K_{i}$ acts as
scalar $\ell_{i}$ for $i=1,2,3,4.$}
\end{definition}

Next we associate the Lie algebra $\mathcal{L}^{*}$ with a specific vertex algebra and its
$\phi$-coordinated modules. We first construct a new Lie algebra.

\begin{definition}\label{New Lie algebra 2}
{\em
Let $\widehat{\mathfrak{L}^{*}}$ be a vector space spanned by $T^{m}\otimes t^{n},E^{r}\otimes t^{s}, K_{i}$ for
$(m,n) , (r,s)\in\Z^{2}, i=1,2,3,4$,
we define
\begin{eqnarray*}
[T^{m}\otimes t^{n}, T^{r}\otimes t^{s}]=0;\;\;\;[K_{i}, \widehat{\mathfrak{L}^{*}}]=0,\;\;\mbox{for}\;\; i=1,2,3,4;
\end{eqnarray*}
\begin{eqnarray*}
&&{}[T^{m}\otimes t^{n}, E^{r}\otimes t^{s}] \nonumber\\
&&{}=(nr-ms)T^{m+r}\otimes t^{n+s-1} + m\delta_{m+r,0}\delta_{n+s+1,0}K_{1}+n\delta_{m+r,0}\delta_{n+s,0}K_{2};
\end{eqnarray*}
\begin{eqnarray*}
&&{}[E^{m}\otimes t^{n}, E^{r}\otimes t^{s}] \nonumber\\
&&{}=(nr-ms)E^{m+r}\otimes t^{n+s-1} + m\delta_{m+r,0}\delta_{n+s+1,0}K_{3}+n\delta_{m+r,0}\delta_{n+s,0}K_{4}.
\end{eqnarray*}}
\end{definition}
It is straightforward to check that these are Lie brackets, so $\widehat{\mathfrak{L}^{*}}$ is a Lie algebra.
 For $(m,n), (r,s)\in\Z^{2}$, we denote $T^{m}\otimes t^{n},E^{r}\otimes t^{s}$ by $(T^{m})_{n},(E^{r})_{s}$,
and we set
\begin{eqnarray}
T^{m}(x)=\sum\limits_{n\in\Z}(T^{m})_{n}x^{-n-1}\;\;\; E^{r}(x)=\sum\limits_{s\in\Z}(E^{r})_{s}x^{-s-1}.      \label{eq:4.4}
\end{eqnarray}
Then the defining relations of $\widehat{\mathfrak{L}^{*}}$ amount to:
\begin{eqnarray}
&&{} [T^{m}(x_{1}),E^{r}(x_{2})]\nonumber\\
&&{}=(m+r)T^{m+r}(x_{2})\frac{\partial}{\partial x_{2}}x_{1}^{-1}\delta\left(\frac{x_{2}}{x_{1}}\right)
         + m\left(\frac{\partial}{\partial x_{2}}T^{m+r}(x_{2})\right)x_{1}^{-1}\delta\left(\frac{x_{2}}{x_{1}}\right) \nonumber\\
&&{}\quad+m\delta_{m+r,0}x_{1}^{-1}\delta\left(\frac{x_{2}}{x_{1}}\right)K_{1}
           +\delta_{m+r,0}\frac{\partial}{\partial x_{2}}x_{1}^{-1}\delta\left(\frac{x_{2}}{x_{1}}\right)K_{2},  \label{eq:4.5}
\end{eqnarray}
and
\begin{eqnarray}
&&{} [E^{m}(x_{1}),E^{r}(x_{2})]\nonumber\\
&&{}=(m+r)E^{m+r}(x_{2})\frac{\partial}{\partial x_{2}}x_{1}^{-1}\delta\left(\frac{x_{2}}{x_{1}}\right)
         + m\left(\frac{\partial}{\partial x_{2}}E^{m+r}(x_{2}\right)x_{1}^{-1}\delta\left(\frac{x_{2}}{x_{1}}\right) \nonumber\\
&&{}\quad+m\delta_{m+r,0}x_{1}^{-1}\delta\left(\frac{x_{2}}{x_{1}}\right)K_{3}
           +\delta_{m+r,0}\frac{\partial}{\partial x_{2}}x_{1}^{-1}\delta\left(\frac{x_{2}}{x_{1}}\right)K_{4}.  \label{eq:4.6}
\end{eqnarray}

Set
$$ \widehat{\mathfrak{L}^{*}}_{\geq 0}= \coprod_{m\in\Z}\left(T^{m}\otimes\C[t]\right)\oplus\coprod_{r\in\Z}\left(E^{r}\otimes \C[t]\right)\oplus\sum\limits_{i=1}^{4}\C K_{i},$$
$$\widehat{\mathfrak{L}^{*}}_{< 0}= \coprod_{m\in\Z}\left(T^{m}\otimes t^{-1}\C[t^{-1}]\right)\oplus\coprod_{r\in\Z}\left(E^{r}\otimes t^{-1}\C[t^{-1}]\right).$$
We see that $\widehat{\mathfrak{L}^{*}}_{\geq 0}$ and $\widehat{\mathfrak{L}^{*}}_{< 0}$ are Lie subalgebras
and $\widehat{\mathfrak{L}^{*}}=\widehat{\mathfrak{L}^{*}}_{\geq 0}\oplus\widehat{\mathfrak{L}^{*}}_{< 0}$
as a vector space.
Let $\ell_{i}\in\C, i=1,2,3,4$, we denote by $\C_{\ell_{1234}}=\C$ the one-dimensional
$\widehat{\mathfrak{L}^{*}}_{\geq 0}$-module
with $\coprod_{m\in\Z}\left(T^{m}\otimes\C[t]\right)\oplus\coprod_{r\in\Z}\left(E^{r}\otimes \C[t]\right)$ acting trivially and
$K_{i}$ acting as $\ell_{i}$ for $i=1,2,3,4.$
Form the induced module
$$V_{\widehat{\mathfrak{L}^{*}}}(\ell_{1234},0)
=U(\widehat{\mathfrak{L}^{*}})\otimes_{U(\widehat{\mathfrak{L}^{*}}_{\geq 0})}\C_{\ell_{1234}}. $$
Set ${\bf 1} =1\otimes 1\in V_{\widehat{\mathfrak{L}^{*}}}(\ell_{1234},0)$,
define a linear operator $\overline{d}$ on $\widehat{\mathfrak{L}^{*}}$
by
$$\overline{d}(K_{i})=0, \;\;\mbox{for}\;\; i=1,2,3,4,$$
$$\overline{d}(T^{m}\otimes t^{n})=-n T^{m}\otimes t^{n-1},\;\;\mbox{and}\;\; \overline{d}(E^{r}\otimes t^{s})=-s E^{r}\otimes t^{s-1}.$$
By Theorem 5.7.1 of \cite{LL}, $V_{\widehat{\mathfrak{L}^{*}}}(\ell_{1234},0)$ is a vertex algebra,
which is uniquely determined by the condition that ${\bf 1}$ is the vacuum vector and
\begin{eqnarray}
 Y((T^{m})_{-1}{\bf 1},x)=T^{m}(x),  \label{eq:4.7}
 \end{eqnarray}
\begin{eqnarray}
 Y((E^{r})_{-1}{\bf 1},x)=E^{r}(x)  \label{eq:4.8}
 \end{eqnarray}
for $(T^{m})_{-1},(E^{r})_{-1}\in\widehat{\mathfrak{L}^{*}} $, $m,r\in\Z$.
Furthermore, $\{(T^{m})_{-1}{\bf 1},(E^{r})_{-1}{\bf 1} \ | \ m,r\in\Z\}$ is a generating subset of $V_{\widehat{\mathfrak{L}^{*}}}(\ell_{1234},0)$.

Similarly, $V_{\widehat{\mathfrak{L}^{*}}}(\ell_{1234},0)$ has its universal property like Remark \ref{universal1} and Remark \ref{universal2}.
As one of our main results of this section, we have:

\begin{thm} Let $W$ be a restricted $\mathcal{L}^{*}$-module of level $\ell_{1234}$.
                  Then there exists a
                  $\phi$-coordinated $V_{\widehat{\mathfrak{L}^{*}}}(\ell_{1234},0)$-module
                 structure $Y_{W}(\cdot,x)$ on $W$, which is uniquely
                 determined by
                 $$Y_{W}((T^{m})_{-1}{\bf 1},x) = T^{m}(x)\;\;\mbox{and}\;\; Y_{W}((E^{m})_{-1}{\bf 1},x) = E^{m}(x)
                   \  \  \  \mbox{ for }m\in\Z.$$
\end{thm}
\begin{proof}Since $\{(T^{m})_{-1}{\bf 1},(E^{r})_{-1}{\bf 1} \ | \ m,r\in\Z\}$ generates $V_{\widehat{\mathfrak{L}^{*}}}(\ell_{1234},0)$
as a vertex algebra, the uniqueness is clear. We now prove the existence.
Set $U_{W}=\{ {\bf 1}_{W}\}\cup\{T_{m}(x),E_{m}(x)\ | \ m\in\Z\}\subset\mathcal{E}(W)$.
From (\ref{eq:4.2}) and (\ref{eq:4.3}), by using Lemma 2.1 of \cite{Li5} we see that
\begin{eqnarray}
(x_{1}-x_{2})^{2} [T_{m}(x_{1}),E_{r}(x_{2})]=0\;\;\mbox{and}\;\;
(x_{1}-x_{2})^{2}[E_{m}(x_{1}), E_{r}(x_{2})]=0.  \label{eq:4.9}
\end{eqnarray}
Then $U_{W}$ is a local subset of $\mathcal{E}(W)$, by Theorem 3.2,
 $U_{W}$ generates a vertex algebra $\langle U_{W}\rangle_{e}$
 under the vertex operator operation $Y_{\mathcal{E}}^{e}$ with $W$ a $\phi$-coordinated module,
 where $Y_{W}(a(x),z) = a(z)$ for $a(x)\in \langle U_{W}\rangle_{e}.$
 With (\ref{eq:4.2}) and (\ref{eq:4.3}), by using the Lemma 4.13 or Proposition 4.14 of \cite{Li5}, we have
 $$T_{m}(x)_{i}^{e}E_{r}(x) = 0   \;\;\mbox{for}\; i\geq 2,$$
  $$T_{m}(x)_{1}^{e}E_{r}(x) =(m+r)T_{m+r}(x)+\delta_{m+r,0}\ell_{2}{\bf 1}_{W},$$
 $$T_{m}(x)_{0}^{e}E_{r}(x) =m\left(x\frac{\partial}{\partial x}T_{m+r}(x)\right)+m\delta_{m+r,0}\ell_{1}{\bf 1}_{W},$$
 and
 $$E_{m}(x)_{i}^{e}E_{r}(x) = 0   \;\;\mbox{for}\; i\geq 2,$$
  $$E_{m}(x)_{1}^{e}E_{r}(x) =(m+r)E_{m+r}(x)+\delta_{m+r,0}\ell_{4}{\bf 1}_{W},$$
 $$E_{m}(x)_{0}^{e}E_{r}(x) =m\left(x\frac{\partial}{\partial x}E_{m+r}(x)\right)+m\delta_{m+r,0}\ell_{3}{\bf 1}_{W}.$$
Then again by Borcherds' commutator formula we have
\begin{eqnarray*}
&&[Y_{\mathcal{E}}^{e}(T_{m}(x),x_{1}),Y_{\mathcal{E}}^{e}(E_{r}(x),x_{2})] \nonumber\\
&=&\sum_{i\geq 0}Y_{\mathcal{E}}^{e}(T_{m}(x)_{i}^{e}E_{r}(x),x_{2})
   \frac{1}{i!}\left(\frac{\partial}{\partial x_{2}}\right)^{i}x_{1}^{-1}\delta\left(\frac{x_{2}}{x_{1}}\right)  \nonumber\\
&=&mY_{\mathcal{E}}^{e}(x\frac{\partial}{\partial x}T_{m+r}(x),x_{2})x_{1}^{-1}\delta\left(\frac{x_{2}}{x_{1}}\right)
    +m \delta_{m+r,0}\ell_{1}{\bf 1}_{W} x_{1}^{-1}\delta\left(\frac{x_{2}}{x_{1}}\right)
                                  \nonumber\\
 &&{} +(m+r)Y_{\mathcal{E}}^{e}(T_{m+r}(x),x_{2}) \frac{\partial}{\partial x_{2}} x_{1}^{-1}\delta\left(\frac{x_{2}}{x_{1}}\right)
       + \delta_{m+r,0}\ell_{2}{\bf 1}_{W} \frac{\partial}{\partial x_{2}} x_{1}^{-1}\delta\left(\frac{x_{2}}{x_{1}}\right) ,
 \end{eqnarray*}
 and
 \begin{eqnarray*}
&&[Y_{\mathcal{E}}^{e}(E_{m}(x),x_{1}),Y_{\mathcal{E}}^{e}(E_{r}(x),x_{2})] \nonumber\\
&=&\sum_{i\geq 0}Y_{\mathcal{E}}^{e}(E_{m}(x)_{i}^{e}E_{r}(x),x_{2})
   \frac{1}{i!}\left(\frac{\partial}{\partial x_{2}}\right)^{i}x_{1}^{-1}\delta\left(\frac{x_{2}}{x_{1}}\right)  \nonumber\\
&=&mY_{\mathcal{E}}^{e}(x\frac{\partial}{\partial x}E_{m+r}(x),x_{2})x_{1}^{-1}\delta\left(\frac{x_{2}}{x_{1}}\right)
    +m \delta_{m+r,0}\ell_{3}{\bf 1}_{W} x_{1}^{-1}\delta\left(\frac{x_{2}}{x_{1}}\right)
                                  \nonumber\\
 &&{} +(m+r)Y_{\mathcal{E}}^{e}(E_{m+r}(x),x_{2}) \frac{\partial}{\partial x_{2}} x_{1}^{-1}\delta\left(\frac{x_{2}}{x_{1}}\right)
       + \delta_{m+r,0}\ell_{4}{\bf 1}_{W} \frac{\partial}{\partial x_{2}} x_{1}^{-1}\delta\left(\frac{x_{2}}{x_{1}}\right) .
 \end{eqnarray*}
Similar as in the proof of Theorem \ref{coordinated mod1} we get
$W$ is a $\phi$-coordinated module
with
$$Y_{W}((T^{m})_{-1}{\bf 1},x) = T_{m}(x)\;\;\mbox{and}\;\; Y_{W}((E^{m})_{-1}{\bf 1},x) = E_{m}(x)
                   \  \  \  \mbox{ for }m\in\Z.$$
\end{proof}

On the other hand, we have:
\begin{thm}
Let $W$ be a $\phi$-coordinated $V_{\widehat{\mathfrak{L}^{*}}}(\ell_{1234},0)$-module.
Then $W$ is a restricted $\mathcal{L}^{*}$-module of level $\ell_{1234}$ with $T_{m}(x)=Y_{W}((T^{m})_{-1}{\bf 1},x)$,
 and $E_{m}(x)=Y_{W}((E^{m})_{-1}{\bf 1},x)$ for $m\in\Z.$
 \end{thm}
\begin{proof}
For $m,r\in\Z,$ since $Y((T^{m})_{-1}{\bf 1},x)=T^{m}(x),Y((E^{m})_{-1}{\bf 1},x)=E^{m}(x)$,
 with (\ref{eq:4.5}) and (\ref{eq:4.6}), by using (\ref{eq:2.1}) we see that
$$(x_{1}-x_{2})^{2}[Y((T^{m})_{-1}{\bf 1},x_{1}),Y((E^{r})_{-1}{\bf 1},x_{2})]=0,$$
$$(x_{1}-x_{2})^{2}[Y((E^{m})_{-1}{\bf 1},x_{1}),Y((E^{r})_{-1}{\bf 1},x_{2})]=0.$$
Note that for $m,r\in\Z, i\geq 0$
\begin{eqnarray*}
&&{}((T^{m})_{-1}{\bf 1})_{i}(E^{r})_{-1}{\bf 1}=(T^{m})_{i}(E^{r})_{-1}{\bf 1}=[T^{m}\otimes t^{i}, E^{r}\otimes t^{-1}]{\bf 1}\nonumber\\
&&{}=(ri+m)T^{m+r}\otimes t^{i-2}{\bf 1}+m\delta_{m+r,0}\delta_{i,0}\ell_{1}{\bf 1}+i\delta_{m+r,0}\delta_{i-1,0}\ell_{2}{\bf 1}.
\end{eqnarray*}
and
\begin{eqnarray*}
&&{}((E^{m})_{-1}{\bf 1})_{i}(E^{r})_{-1}{\bf 1}=(E^{m})_{i}(E^{r})_{-1}{\bf 1}=[E^{m}\otimes t^{i},E^{r}\otimes t^{-1}]{\bf 1}\nonumber\\
&&{}=(ri+m)E^{m+r}\otimes t^{i-2}{\bf 1}+m\delta_{m+r,0}\delta_{i,0}\ell_{3}{\bf 1}+i\delta_{m+r,0}\delta_{i-1,0}\ell_{4}{\bf 1}.
\end{eqnarray*}
By Proposition 5.9 of \cite{Li4}, we have
\begin{eqnarray*}
&&{}[Y_{W}((T^{m})_{-1}{\bf 1},x_{1}),Y_{W}((E^{r})_{-1}{\bf 1},x_{2})] \nonumber\\
&&{}=\mbox{Res}_{x_{0}}x_{1}^{-1}\delta\left(\frac{x_{2}e^{x_{0}}}{x_{1}}\right)
    x_{2}e^{x_{0}}Y_{W}(Y(T^{m},x_{0})E^{r},x_{2})\nonumber\\
&&{}=mY_{W}((T^{m+r})_{-2}{\bf 1},x_{2})\delta\left(\frac{x_{2}}{x_{1}}\right)
   +m\delta_{m+r,0}\delta\left(\frac{x_{2}}{x_{1}}\right)\ell_{1}{\bf 1}_{W}
                      \nonumber\\
&&{}\quad\;+(m+r)Y_{W}((T^{m+r})_{-1}{\bf 1},x_{2})\left(x_{2}\frac{\partial}{\partial x_{2}}\right)\delta\left(\frac{x_{2}}{x_{1}}\right)
           +\delta_{m+r,0}\left(x_{2}\frac{\partial}{\partial x_{2}}\right)\delta\left(\frac{x_{2}}{x_{1}}\right)\ell_{2}{\bf 1}_{W},
\end{eqnarray*}
and
\begin{eqnarray*}
&&{}[Y_{W}((E^{m})_{-1}{\bf 1},x_{1}),Y_{W}((E^{r})_{-1}{\bf 1},x_{2})] \nonumber\\
&&{}=\mbox{Res}_{x_{0}}x_{1}^{-1}\delta\left(\frac{x_{2}e^{x_{0}}}{x_{1}}\right)
    x_{2}e^{x_{0}}Y_{W}(Y(E^{m},x_{0})E^{r},x_{2})\nonumber\\
&&{}=mY_{W}((E^{m+r})_{-2}{\bf 1},x_{2})\delta\left(\frac{x_{2}}{x_{1}}\right)
  +m\delta_{m+r,0}\delta\left(\frac{x_{2}}{x_{1}}\right)\ell_{3}{\bf 1}_{W}
      \nonumber\\
&&{}\quad\;+(m+r)Y_{W}((E^{m+r})_{-1}{\bf 1},x_{2})\left(x_{2}\frac{\partial}{\partial x_{2}}\right)\delta\left(\frac{x_{2}}{x_{1}}\right)
           +\delta_{m+r,0}\left(x_{2}\frac{\partial}{\partial x_{2}}\right)\delta\left(\frac{x_{2}}{x_{1}}\right)\ell_{4}{\bf 1}_{W},
\end{eqnarray*}
then we can prove as Theorem \ref{coordinated mod2} that $W$ is a restricted $\mathcal{L}^{*}$-module of level $\ell_{1234}$.
\end{proof}

\end{document}